\documentclass[11pt]{article}

\usepackage{graphicx}
 \usepackage{mathptmx} 
\usepackage{epsfig,tabularx}
\usepackage{mathtools, esdiff}
\usepackage[subnum]{cases}
\usepackage{enumerate}
\usepackage{amsmath,amssymb,amsfonts,amsthm}
\usepackage{color,physics}
\usepackage[style=numeric-comp,url=false,doi=false, bibencoding=utf8,giveninits=true]{biblatex}
\DeclareNameAlias{default}{family-given}
\bibliography{references}

 \usepackage[bookmarks=true,
bookmarksnumbered=true, breaklinks=true,
pdfstartview=FitH, hyperfigures=false,
plainpages=false, naturalnames=true,
colorlinks=true,
pdfpagelabels]{hyperref}

\newcommand{\R}{\mathbb{R}}

\newcommand{\s}{\mathbb{S}}
\newcommand{\supp}{\text{supp}}

\newcommand{\conbod}{\mathcal{K}^n}

\newcommand{\vol}{\text{\rm Vol}}

\newtheorem{theorem}{Theorem}[section]
\newtheorem{definition}[theorem]{Definition}
\newtheorem{lemma}[theorem]{Lemma}
\newtheorem{corollary}[theorem]{Corollary}
\newtheorem{remark}[theorem]{Remark}
\newtheorem{proposition}[theorem]{Proposition}
\newtheorem{example}[theorem]{Example}

\begin{document}

	\title{General measure extensions of projection bodies 
	}
	\author{Dylan Langharst, Michael Roysdon, and Artem Zvavitch \footnote{The first and third named authors were supported in part by the U.S. National Science Foundation Grant DMS-2000304 and the United States - Israel Binational Science Foundation (BSF) Grant 2018115. The second named author was supported in part by Zuckerman STEM Leadership Program and in part by the European Research Council (ERC) under the European Union’s Horizon 2020 research and innovation programme (grant agreement No 770127). MSC 2020 Classification: 52A39, 52A40;  Secondary: 28A75
}}

\maketitle

\begin{abstract}
The inequalities of Petty and Zhang are affine isoperimetric-type inequalities providing sharp bounds for $\vol^{n-1}_{n}(K)\vol_n(\Pi^\circ K),$ where  $\Pi K$ is a projection body of a convex body $K$. In this paper, we present a number of generalizations of Zhang's inequality to the setting of arbitrary measures.	In addition, we introduce extensions of the projection body operator $\Pi$ to the setting of arbitrary measures and functions, while providing associated inequalities for this operator; in particular, Zhang-type inequalities. Throughout, we apply shown results to the standard Gaussian measure. Under a certain restriction on $K,$ we use our results to obtain a reverse isoperimetric inequality for the measure of $\partial K.$
\end{abstract}

\section{Introduction}\label{Introduction}
 We consider the $n$-dimensional Euclidean space $\R^n$ equipped with its usual structure induced by $\langle x,y\rangle$, the inner product of two vectors $x, y \in \mathbb R^n$, and denote by $|x|$ the length of a vector $x \in \mathbb{R}^n$. The closed unit ball in $\R^n$ is denoted by $B_2^n$, and its boundary by $\s^{n-1}$. Moreover, for any set in $A \subset \R^n$, we denote its boundary by $\partial A$.  We write $\vol_m(A)$ for the $m$-dimensional Lebesgue measure (volume) of a measurable set $A \subset \R^n$, where $m = 1, . . . , n$ is the dimension of the minimal affine space containing $A$. A convex body is a convex, compact set with nonempty interior. We denote by $\conbod$ the set of convex bodies in $\R^n$. For  $K \in \conbod $ the support function of $K$, $h_K:\R^n \to \R$, is defined as  $h_K(x)=\sup \{\langle x,y\rangle \colon y \in K \}$. For $K,L \in \conbod$, their Minkowski sum is defined as the set $K+L=\{ x + y: x\in K, y\in L\}.$ For the purposes of Minkowski addition, a singleton set $\{\xi\}$ is simply written as $\xi$.  From the definition,
  $$h_{K+L}(\theta) = h_{K}(\theta)+h_{L}(\theta)$$ and, in particular,  we have for a vector $\xi\in\R^n$, that $h_{K+\xi}(\theta)=h_K(\theta)+\langle\xi,\theta\rangle.$ An important subset of this space is the set of convex bodies containing the origin in their interior, $\conbod_0$. For $K\in\conbod_0$, the dual body of $K$ is given by $$K^\circ=\left\{x\in \R^n: h_K(x) \leq 1\right\}.$$A convex body $K$ is centrally symmetric, or just symmetric, if $K=-K$. 
Furthermore, for $K \in \conbod$, we denote the orthogonal projection of $K$ onto a linear subspace $H$ as $P_H K$;  the \textbf{projection body} of $K$, denoted $\Pi K$, is defined via its support function
\begin{align*}
    h_{\Pi K}(\theta)=\vol_{n-1}(P_{\theta^{\perp}}K),
\end{align*}
where  $\theta^{\perp}=\{x\in\R^n:\langle \theta,x\rangle =0\}$ is the subspace orthogonal to $\theta \in \s^{n-1}$.  

We refer the reader to \cite{gardner_book,Gr,Ko,KoY,Sh1},  as well as Section~\ref{preliminaries} below, for more definitions and properties of convex bodies and corresponding functionals. In particular, Cauchy's projection formula shows that $\Pi K$ is a symmetric convex body. Relations between a convex body $K$ and its polar projection body $(\Pi K)^\circ=\Pi^\circ K$ have been studied extensively; in particular, the following bounds have been established: for any $K\in\conbod$, one has, with $\kappa_m=\vol_m(B_2^m),$
\begin{equation}\label{e:Zhang_ineq}
\frac{1}{n^n} {2n \choose n}\leq\vol_n(K)^{n-1}\vol_n(\Pi^\circ K)\leq \left(\frac{\kappa_n}{\kappa_{n-1}}\right)^n.
\end{equation}
The right-hand side of \eqref{e:Zhang_ineq} is \textbf{Petty's inequality} which was proven by Petty in 1971 \cite{petty}; equality occurs in Petty's inequality if, and only if, $K$ is an ellipsoid. The left-hand side of \eqref{e:Zhang_ineq} is known as \textbf{Zhang's inequality}. It was proven by Zhang in 1991 \cite{Zhang}. Equality holds in Zhang's inequality if, and only if, $K$ is a simplex (convex hull of $n+1$ affinely independent points).

Over the last two decades, a number of classical results in convex geometry have been extended to the setting of arbitrary measures. This includes extensions of the Busemann-Petty and slicing problems \cite{BP,K2,K3,K4,KK,KLi,KZ,KPZv,Zv1,Zv2}, works on the surface area measure \cite{Ball3,Naz,GAL2,GAL3,GAL4,GAL5}, and general measure extensions of the projection body of a convex body \cite{GAL}.

One of the goals of this paper is to consider a variety of generalizations of Zhang's inequality to the setting of different measures, akin to the undertaking by Livshyts in \cite{GAL} for surface area. We denote by $\Lambda$ the set of all locally finite, regular Borel measures $\mu$ whose Radon-Nikodym derivative, or density, is from $\R^n$ to $\R^{+}:=[0,\infty)$, i.e, 
\[
\mu \in \Lambda \iff \frac{d\mu(x)}{dx} = \phi(x), \text{ with } \phi \colon \R^n \to \R^+, \phi\in L^1_{\text{loc}}(\R^n). 
\]
We denote the Lebesgue measure itself as $\lambda.$ In addition, if $\nu\in\Lambda$ with density $\varphi$, then, $\nu$ is said to be radially non-decreasing measure if, for every $x\in\R^n$ and $t\in[0,1]$, we have $\varphi(tx)\leq\varphi(x)$. We denote $\Lambda_{\text{rad}}$ as the subset of $\Lambda$ consisting of radially non-decreasing measures which have densities continuous on $\R^n\setminus\{0\}$.

The proof of Zhang's inequality, as presented in \cite{GZA} made critical use of the covariogram function, which appears often in the literature. We recall its definition (see also \cite{Sh1}).  
\begin{definition}
	For $K\in\conbod$ the \textbf{covariogram} of $K$ is given by
	\begin{equation}\label{e:covario}
	    g_K(x)=\vol_n\left(K\cap(K+x)\right)=(\chi_K\star\chi_{-K})(x),
	\end{equation}
where $(f\star g)(x)= \int_{\R^n} f(y)g(x-y)dy$ is a convolution of functions $f,g:\R^n \to \R$ and
	$\chi_K(x)$ is the characteristic function of $K$. 
	\end{definition}The support of $g_K(x)$ is the difference body of $K$, given by 
	\begin{equation}\label{e:differnce}
	    DK=\{x:K\cap(K+x)\neq \emptyset\}=K+(-K).
	\end{equation}
One of the crucial steps in the proof of Zhang's inequality in \cite{GZA} was to calculate the brightness of a convex body $K$, that is the derivative of the covariogram of $K$ in the radial direction, evaluated at $r=0$. Being defined as a convolution of characteristic functions, the covariogram inherits the $1/n$ concavity property of the Lebesgue measure. The proofs of these facts can be found in \cite{Sh1}. 

We say a set $L$ with $0 \in \text{int}(L)$ is star-shaped if every line passing through the origin crosses the boundary of $L$ exactly twice. We say $L$ is a star body 
if it is a compact, star-shaped set whose radial function $\rho_L:\R^n\setminus \{0\} \to \R,$ given by $\rho_L(y)=\sup\{\lambda:\lambda y\in L\},$ is continuous. Furthermore, for $K\in\conbod_0,$ the \textit{Minkowski functional} of $K$ is defined to be $\|y\|_K=\rho^{-1}_K(y)=\inf\{r>0:y\in rK\}.$ The Minkowski functional $\|\cdot\|_K$ of $K\in\conbod_0$ is a norm on $\R^n$ if $K$ is symmetric. If $x\in\R^n$ is so that $L-x$ is a star body, then the generalized radial function of $L$ at $x$ is defined by $\rho_L(x,y)=\rho_{L-x}(y)$. Note that for every $K\in\conbod,$ $K-x$ is a star body for every $x\in\text{int}(K)$.  Gardner and Zhang \cite{GZA} defined the \textit{radial pth mean bodies} of a convex body $K$ as the star body whose radial function is given by, for $\theta\in\s^{n-1},$ 
\begin{equation}
    \rho_{R_p K}(\theta)=\left(\frac{1}{\vol_n(K)}\int_K \rho_K(x,\theta)^p dx\right)^\frac{1}{p}.
    \label{pth}
\end{equation}
A priori, the above is valid for $p>0$.
But also, by appealing to continuity, Gardner and Zhang were able to define
$\rho_{R_\infty K}(\theta)=\max_{x\in K}\rho_K(x,\theta)=\rho_{DK}(\theta)$ and
$\rho_{R_{0} K}(\theta)=\exp \left(\frac{1}{\vol_n(K)} \int_{K} \log \rho_{K}(x, \theta) d x\right).$ The fact that
$$\int_K \rho_K(x,\theta)^pdx=p\int_K\int_0^{\rho_K(x,\theta)}r^{p-1}drdx=p\int_0^{\rho_{DK}(\theta)}\left(\int_{K\cap (K+r\theta)}dx\right)r^{p-1}dr=p\int_0^{\rho_{DK}(\theta)}g_K(r\theta)r^{p-1}dr$$
for $p>0$ shows that each $R_p K$ is a symmetric convex body ($p=0$ follows by continuity), as integrals of the above form are radial functions of certain symmetric convex bodies (see \cite[Theorem 5]{LBall} for $p\geq 1$ and \cite[Corollary 4.2]{GZA}). By using Jensen's inequality, one has for $0\leq p\leq q\leq \infty$
\begin{equation}
R_{0} K \subseteq R_{p} K \subseteq R_{q} K \subseteq R_{\infty} K=D K.
\label{eq:radial_set_inc}
\end{equation}
To extend to the range to $p\in (-1,0),$ Gardner and Zhang defined another family of star bodies depending on $K\in\conbod$, the \textit{spectral pth mean bodies} of $K,$ denoted $S_p K.$ However, to apply Jensen's inequality, they had to assume additionally that $\vol_n(K)=1.$ To avoid this assumption, we change the normalization and define $S_p K$ as the star body whose radial function is given by, for $p\in [-1,\infty),$
$$\rho_{S_p K}(\theta)=\left(\int_{P_{\theta^\perp}K} X_\theta K(y)^{p}\left(\frac{X_\theta K(y)dy}{\vol_n(K)}\right)\right)^{1/p},$$
where $X_\theta K$ is the \textit{X-ray} of $K$ in the direction $\theta\in\s^{n-1}$ (see \cite[Chapter 1]{gardner_book} for a precise definition of $X_\theta K$, but note that $\int_{P_{\theta^\perp}K}\frac{X_\theta K(y)dy}{\vol_n(K)}=1$), $\rho_{S_{\infty} K}(\theta) = \max_{y\in \theta^\perp}X_\theta K(y)=\rho_{DK}(\theta),$
$\rho_{S_0 K}(\theta)=\exp\left(\int_{P_{\theta^\perp}K}\log(X_\theta K(y))\frac{X_\theta K(y)dy}{\vol_n(K)}\right),$
and $$\rho_{S_{-1}K}(\theta)=\vol_n(K)\vol_{n-1}(P_{\theta^\perp}K)^{-1}=\vol_n(K)\rho_{\Pi^\circ K}(\theta).$$

Therefore, from Jensen's inequality, we obtain, for $-1\leq p \leq q \leq \infty,$
\begin{equation}\vol_n(K)\Pi^\circ K=S_{-1}K\subseteq S_p K \subseteq S_q K \subseteq S_{\infty}K=DK.\label{eq:spectral}\end{equation}
The fact that, for $p>-1,$
$$\frac{1}{p+1}\int_{P_{\theta^\perp} K}X_\theta K(y)^{p+1}dy=\int_{P_{\theta^\perp} K}\int_0^{X_\theta K(y)}r^pdrdy=\int_K \rho_K(x,\theta)^p dx$$
shows $S_0 K = e R_0 K$, $S_p K = (p+1)^{1/p}R_p K, \, p>0, $ and that we can analytically continue $R_p K$ to $p\in (-1,0)$ by $R_p K:=(p+1)^{-1/p}S_p K$ (note that due to the alternate normalization of $S_p K,$ this is different than in \cite[Theorem 2.2]{GZA}; in both instances,  it is unknown if $R_p K$ and $S_p K$ are convex for $p\in (-1,0)$). Gardner and Zhang then obtained a reverse of \eqref{eq:radial_set_inc}. They accomplished this by showing \cite[Theorem 5.5]{GZA}, for $0\leq p \leq q < \infty,$ that
\begin{equation}
    \label{e:set_inclusion}
    D K \subseteq c_{n, q} R_{q} K \subseteq c_{n, p} R_{p} K \subseteq n \vol_n(K) \Pi^{\circ} K,
\end{equation}
where $c_{n, p}$ are constants defined as $$
c_{n, p}=(n B(p+1, n))^{-1 / p} \text{ for $p>0$ and} \; c_{n, 0}=\lim _{p \rightarrow 0}(n B(p+1, n))^{-1 / p}=\exp \left(\sum_{k=1}^{n} 1 / k\right),
$$ $B(x,y)=\frac{\Gamma(x)\Gamma(y)}{\Gamma(x+y)}$ is the standard Beta function, and  $\Gamma(x)$ is the standard Gamma function. There is equality in each inclusion in \eqref{e:set_inclusion} if, and only if, $K$ is a simplex. Notice that the far-left and far-right of \eqref{eq:spectral} and \eqref{e:set_inclusion} yields $\vol_n(K)\Pi^\circ K \subseteq DK \subseteq n\vol_n(K)\Pi^\circ K$. This shows the geometric distance between $\vol_n(K)\Pi^\circ K$ and $DK$ is at most $n$. Note that $DK$ and $K$ are also related to an affine invariant quantity. For $K\in\conbod,$
	\begin{equation}\label{e:roger_shep}
	2^n\leq \frac{\vol_n(DK)}{\vol_n(K)}\leq{2n \choose n},
	\end{equation}
	where the left-hand side follows from the Brunn-Minkowski inequality (see the discussion on concavity of measures below) and has equality if, and only if, $K$ is symmetric. The right-hand side was proven by Rogers and Shephard  \cite{RS1}, with equality if, and only if, $K$ is a simplex. In \cite{AHNRZ} the existence of various Rogers-Shephard type inequalities for general measures $\mu\in\Lambda$ was established. Their results were in terms of the \textbf{translated-average} of $K$ with respect to $\mu$, which we recall is, for a Borel measure $\mu$:
\begin{equation}\label{e:mu_average}
	    \mu_\lambda(K):=\frac{1}{\vol_n(K)}\int_K \mu(y-K)dy=\frac{1}{\vol_n(K)}\int_{K} \int_{\mathbb{R}^{n}} \chi_{K}(y-x) d \mu(x) dy =\frac{1}{\vol_n(K)}\int_{\mathbb{R}^{n}} g_{K} (x) d \mu(x),
	\end{equation}
	where the second equality follows from the fact that $x\in y-K$ if, and only if, $y-x\in K,$ and hence $\chi_{y-K}(x)=\chi_{K}(y-x);$ the last equality follows from Fubini's theorem and \eqref{e:covario}. The quantity $\mu_\lambda(K)$, obeys the following reverse Rogers-Shephard type inequality:
	\begin{equation}
	\label{eq:mu_rever_rog}\mu_\lambda(K)=\frac{1}{\vol_n(K)}\int_K \mu(y-K)dy=\frac{1}{\vol_n(K)}\int_{DK} \vol_n(K\cap (K+x))d\mu(x)\leq \mu(DK).
	\end{equation}
	The following Rogers-Shephard inequality for radially decreasing measures was proved in \cite{AHNRZ}:
\begin{proposition}
\label{prop:RST}
Let $K \in \mathcal{K}^{n}$. Let $\nu\in\Lambda$ have radially decreasing density. Then
$$
\nu(K-K) \leq {{2n} \choose n}\nu_{\lambda}(K).
$$
Moreover, if $\phi$ is continuous at the origin, then equality holds if, and only if, $\nu$ is a constant multiple of the Lebesgue measure on $K-K$ and $K$ is a simplex.
\end{proposition} 
	This result is part of a collection of Rogers-Shephard and reverse Rogers-Shephard type inequalities, examined further in \cite{A,AAGJV,AlGMJV,Co,Ro}. In Section~\ref{generalizations}, we discuss analogues of Zhang's inequality for measures with positive density.  These results are in terms of the quantity $\mu_\lambda.$ From the definition, one can already obtain a result: \eqref{eq:mu_rever_rog} and \eqref{e:set_inclusion} yields 
	\begin{equation}\label{eq:av}
	    \mu_\lambda(K)\leq \mu(DK) \leq \mu\left(n\vol_n(K)\Pi^{\circ} K\right).
	\end{equation}
Comparing  this to the left-hand side of \eqref{e:Zhang_ineq}, we see that we would be missing the constant ${{2n}\choose n}$ in the case of $\mu=\lambda$. We show, however, in Theorem~\ref{t:weak_zhang} that \eqref{eq:av} is sharp over $\Lambda.$ Thus, our first main result is Theorem~\ref{th:zhnd}, where we  prove the following analogue of Zhang's inequality for those measures in $\Lambda_\text{rad},$ that is those radially non-decreasing measures with continuous densities on $\R^n\setminus\{0\}$:
for each  $\nu\in\Lambda_{\text{rad}}$  and  all $K\in\conbod,$ we have
\begin{align*}
    {2n \choose n}\nu_\lambda(K)\leq\nu(n\vol_n(K)\Pi^\circ K).
\end{align*}
	
Projection bodies have been extended to the larger setting of functions. Let $f$ be a function in the Sobolev space $W^{1,1}(\R^n)$ (see \cite{Evans}). In \cite{LYZSob}, Lutwak, Yang and Zhang proved the existence of a unique, centrally symmetric convex body generated by $f$, denoted $\langle f \rangle$ and called, in \cite{LUD}, the {\it LYZ body}. From here, the projection body $\Pi \langle f \rangle$ is the convex body whose support function is given by
\begin{equation}\label{e:LYZ_body}
h_{\Pi \langle f \rangle}(\theta)=\frac{1}{2}\int_{\R^n}|\langle \nabla f(x),\theta\rangle|dx.
\end{equation}
An upper bound for the volume of the polar projection LYZ body of a $W^{1,1}(\R^n)$ function $f$ was established by Zhang \cite{zsob} in what is now known as the Sobolev-Zhang inequality. Specifically, if $\|\cdot\|_p$ denotes the $L^p$ norm of a function with respect to $\lambda,$ then $\|f\|_{\frac{n}{n-1}}^n\vol_n(\Pi^\circ \langle f \rangle) \leq \left(\frac{\kappa_n}{\kappa_{n-1}}\right)^n,$ with equality if $f$ is a multiple of a characteristic function of an ellipsoid. This was extended to functions of bounded variation on $\R^n$ by Wang in \cite{wsob}. A lower-bound was also established \cite{ABG} for when $f$ is log-concave.

In Section~\ref{super_proj}, we consider a different generalization of projection bodies to measures in $\Lambda$ and functions. For $\mu\in\Lambda$ with density $\phi,$ we define the measure-dependent projection body of a convex body $K$ as
$$h_{\Pi_\mu K}(\theta)=\frac{1}{2}\int_{\partial K}|\langle\theta,n_K(y) \rangle|\phi(y) dy,$$
where $n_K(y)$ is a normal vector  to  $y \in \partial K$, and is well-defined almost everywhere. These were previously studied in \cite{GAL} for centrally symmetric $K\in\conbod_0$ under the notation $\frac{1}{n}p_{\mu,K}(\theta,1).$ In Section~\ref{super_proj}, we obtain the following isoperimetric inequality relating $\mu(\partial K)$ and $\vol_n(\Pi^\circ_\mu K)$ for $K\in\conbod$ and $\mu\in\Lambda$ such that $\mu(\partial K)\in (0,\infty):$
\begin{equation*}
    \left(\frac{n\kappa_n}{\kappa_{n-1}}\right)^n\kappa_n\leq \mu^n(\partial K)\vol_n(\Pi_\mu^\circ K).
\end{equation*}

In Section~\ref{sec:lem}, we state various lemmas concerning the behaviour of $K\cap (K+x)$ as $|x|$ approaches zero from a fixed direction. In Section~\ref{sec:res}, we prove results for $\Pi_\mu K$. We define the $\mu$-covariogram of $K\in\conbod$:
$$g_{\mu,K}(x)=\int_{K} \chi_K(y-x)\phi(y)dy=\mu(K\cap(K+x)).$$
This measure dependent covariogram is reminiscent, but distinct, from the recent construction of covariograms of convex bodies that depend on valuations by Averkov and Bianchi \cite{AB15}. 

One form of the main results from Section~\ref{sec:res} is showing that, for a symmetric convex body $K$, and an even measure $\mu$ with density Lipschitz on a domain containing $K$, one has for every $\theta\in\s^{n-1}$ that
$$\diff{g_{\mu,K}(r\theta)}{r}\bigg|_{r=0}=-h_{\Pi_{\mu} K}(\theta).$$
The result for non-symmetric $K,$ and $\mu$ not necessarily even, is also shown. We will review properties of Lipschitz functions in Section~\ref{preliminaries} below. To state our next result, we consider measures with different concavity conditions. A measure $\mu\in\Lambda$ is said to be $F$-concave on a class $\mathcal{C}$ of compact Borel subsets of $\R^n$ if there exists a continuous, invertible function $F:(0,\mu(\R^n))\to (-\infty,\infty)$ such that, for every pair $A,B \in \mathcal{C}$  and every $t \in [0,1]$, one has
  $$\mu(t A +(1-t)B) \geq F^{-1}\left(tF(\mu(A)) +(1-t)F(\mu(B))\right).$$
 When $F(x)=x^s, s > 0$ this can be written as
	$$\mu(t A +(1-t)B)^{s}\geq t\mu(A)^s +(1-t)\mu(B)^{s},$$
 and we say $\mu$ is $s$-concave. When $s=1$, we merely say the measure is concave.	In the limit as $s\rightarrow 0$, we obtain the case of log-concavity:
	$$\mu(t A +(1-t)B)\geq\mu(A)^{t}\mu(B)^{1-t}.$$
	 The classical Brunn-Minkowski inequality  (see for example \cite{Stein}) asserts the  $1/n$-concavity of the Lebesgue measure on the class of all compact subsets of $\R^n$. 

An immediate application of the results of Section~\ref{sec:res} is that we are able to obtain the following set inclusion: given a symmetric $K\in\conbod_0$ and $\mu$ an even measure that is $F$-concave, where $F$ is a non-negative, increasing and differentiable function, one has 
$$DK\subseteq \frac{F(\mu(K))}{F^\prime(\mu(K))}\Pi^\circ_{\mu}K.$$
Throughout this paper, we will state applications of proven results to the standard Gaussian measure $\R^n,$ denoted: $$d\gamma_n(x)=\frac{1}{(2\pi)^{\frac{n}{2}}}  e^{-\frac{|x|^2}{2}} dx.$$
Various concavity conditions have been established for $\gamma_n$, for example see \cite{PT,KL,GZ,EHR1,EHR2} for more on these topics. We conclude Section~\ref{sec:res} by proposing an alternative definition of the covariogram function to obtain results from the concavity of the Gaussian measure implied by the dimensional Brunn-Minkowski inequality for symmetric convex bodies, recently proven in \cite{EM}.

In Section~\ref{sec:functions}, we define the projection body of non-negative $f\in L^1(\mu,\partial K)$ where $\mu\in\Lambda$ with density $\phi$ and $K\in\conbod$, denoted $\Pi_{\mu,K} f$, to be the convex body whose support function is given by $$h_{\Pi_{\mu,K} f}(\theta)=\frac{1}{2}\int_{\partial K}|\langle\theta,n_K(y) \rangle| f(y)\phi(y) dy.$$ 
Furthermore, for a function $f\in L^1(\mu,K)$, the $\mu$-covariogram of $f$ is given by
$$g_{\mu,f}(K,r\theta)=\int_{K\cap(K+r\theta)}f(y-r\theta)\phi(y)dy.$$
Geometrically, the non-negative function $f$ is acting on $K+r\theta$ and the function $\phi$ is acting on $K.$ The main result of Section~\ref{sec:functions} can be stated in the following way: for $f\in L^1(\mu,\partial K)\cap L^1(\mu,K)$ that is bounded on a symmetric $K\in\conbod_0$ and differentiable almost everywhere, and even $\mu\in\Lambda$ with density Lipschitz on a domain containing $K$, one has
$$\diff{g_{\mu,f}(K,r\theta)}{r}\bigg|_{r=0}=-h_{\Pi_{\mu,K} f}(\theta).$$

In Sections~\ref{sec:log} and \ref{super_zhang_sec}, we prove results for these measure dependent projection bodies, that is for $\Pi_{\mu,K} f$ and $\Pi_\mu K$. For Section~\ref{sec:log}, the main result is Theorem~\ref{Qf_theorem}, which, in particular, encapsulates log-concave measures: let $f:\R^n\rightarrow \R^{+}$ be an even, bounded and differentiable almost everywhere function in $L^1(K)\cap L^1(\mu,\partial K)$ for some symmetric $K\in\conbod_0$, and $\mu\in\Lambda$ with an even, locally Lipschitz density $\phi$. Now, suppose $Q:(0, \infty)\rightarrow \R$ is an invertible, differentiable, increasing function such that $\lim_{r\to 0^+}Q(r)\in [-\infty,\infty)$ and $\mu$ is $Q$-concave. Then, if $Q^\prime(\|f\|_{L^1(\mu,K)})\neq 0$ and $Q \circ g_{\mu,f}(K,\cdot)$ is concave, one has the following: 
\begin{align*}
    \|f\|_{L^1(K)}\leq \frac{n\vol_n\left(\Pi_{\mu,K}^\circ f\right)}{\mu(K)\left[Q^\prime(\|f\|_{L^1(\mu,K)})\right]^{n}}
    \int_0^{\infty}Q^{-1}\left(Q\left(\|f\|_{L^1(\mu,K)}\right)-t\right)t^{n-1}dtd\theta,
\end{align*}
where
\begin{equation*}
    \|f\|_{L^1(\mu,K)}:=\int_{K}|f(y)|d\mu(y)=g_{\mu,f}(K,0)
\end{equation*}
is the $L^1\left(\mu,K\right)$ norm of a function $f$ over $K$. To make the above inequality more clear, we consider the case when $f=\chi_K$ is the characteristic function of a symmetric convex body $K$:
$$\vol_n(K)\leq \frac{n\vol_n\left(\Pi_{\mu}^\circ K\right)}{\mu(K)(Q^\prime(\mu(K)))^n} \int_0^{\infty}Q^{-1}\left(Q(\mu(K))-t\right)t^{n-1}dt.$$
In particular, when $Q(x)=\log x$ this inequality becomes $\vol_n(K) \leq n!\mu^{n}(K) \vol_n\left(\Pi_{\mu}^\circ K\right).$ We also show the results for more general $f,$ $K,$ and $\mu$, but they are more complicated and we will not reproduce them here.

In Section~\ref{super_zhang_sec}, we present another extension of Zhang's inequality to the setting of measure dependent projection bodies. Our results are in terms of the $\nu$-translation of an integrable function $f\in L^1(\mu,K)$, $K$ a compact set, averaged with respect to $\mu$:
$$\nu_{\mu}(f,K)=\frac{1}{\|f\|_{L^1(\mu,K)}}\int_{DK}g_{\mu,f}(K,x)d\nu(x).$$
For simplicity, we will merely state here the results for when $\mu$ and $f$ are even. Suppose $F$ is an increasing, differentiable, and invertible function such that $F(0)=0$ and $\mu$ is $F$-concave with locally Lipschitz density. Then, for symmetric $K\in\conbod_0$ and $f$ that is bounded and differentiable almost everywhere, one has
\begin{align*}
    \|f\|_{L^1(\mu,K)}\nu_{\mu}(f)\leq n \nu\left(\frac{F(\|f\|_{L^1(\mu,K)})}{F^\prime(\|f\|_{L^1(\mu,K)})}\Pi_{\mu,K}^\circ f\right)
    \int_{0}^1F^{-1}\left(F(\|f\|_{L^1(\mu,K)})t\right)(1-t)^{n-1}dt.
\end{align*}

We conclude with an application; in Section~\ref{sec:iso}, we obtain two reverse isoperimetric inequalities for $\mu(\partial K)$ under some constraint on $K$ and $\mu$. 

\noindent \textbf{Acknowledgements} We would like to thank the anonymous referees for the many helpful comments, greatly improving the structure and presentation of the ideas throughout this paper. We also thank Christos Saroglou for suggesting the application in Section~\ref{sec:iso}.

\section{Preliminaries} \label{preliminaries}

 In this section we will further collect basic definitions and formulas. For $K$ a star body, the volume of $K$ can be obtained using its radial function:
	\begin{equation}\label{e:volume}
	    \vol_n(K)=\int_{\s^{n-1}}\int_{0}^{\rho_K(\theta)}r^{n-1}drd\theta=\frac{1}{n}\int_{\s^{n-1}}\rho^n_K(\theta)d\theta.
	\end{equation}
	Moreover, if $K\in\conbod_0,\,\rho_K(x)=h^{-1}_{K^\circ}(x),$ and hence
	\begin{equation}\label{e:polar_volume}
	    \vol_n(K^\circ)=\frac{1}{n}\int_{\s^{n-1}}\rho^n_{K^\circ}(\theta)d\theta=\frac{1}{n}\int_{\s^{n-1}}h_{K}^{-n}(\theta)d\theta.
	\end{equation} A convex body $K\in\conbod$ can also be studied through its surface area measure: for every Borel $A \subset \s^{n-1},$ one has $$S_K(A)=\mathcal{H}^{n-1}(n^{-1}_K(A)),$$ where $\mathcal{H}^{n-1}$ is the $(n-1)$-dimensional Hausdorff measure and $n_K:\partial K \rightarrow \s^{n-1}$ is the Gauss map, which associates an element $y$ of $\partial K$ with its outer unit normal. For almost all $x\in\partial K$, $n_K(x)$ is well-defined (i.e. $x$ has a single outer unit normal). Since the set $N_K=\{x\in\partial K: n_K(x) \text{ is not well-defined}\}$ is of measure zero, we will continue to write $\partial K$ in place of $\partial K\setminus N_K,$ without any confusion. Using $S_K$, we can write (see \cite[Section 5.1]{Sh1}): $$\vol_n(K)=\frac{1}{n}\int_{\s^{n-1}}h_K(u)dS_K(u).$$ 
	
	We say that a convex body $K$ is {\bf strictly convex} if its boundary does not contain a line segment, in which case the Gauss map is an injection between $\partial K\setminus N_K$ and $\s^{n-1}$. The Gauss map is related to the support function. Fix $u\in\s^{n-1}$; then, $\nabla h_K(u)$ exists if, and only if, $n_K^{-1}(u)$ is a single point $x\in\partial K,$ and, furthermore, $\nabla h_K(u)=x=n^{-1}_K(u)$ \cite[Corollary 1.7.3]{Sh1}. Hence, $K$ is strictly convex if, and only if, $h_K\in C^1$ \cite[Page 115]{Sh1}. If $h_K$ is
 of class $C^2$, then one has $dS_K(\theta)$ is absolutely continuous with respect to the spherical Lebesgue measure \cite[Corollary 2.5.3]{Sh1}: $$dS_K(\theta) = f_K(\theta) d\theta,$$
	where $f_K \colon \s^{n-1} \to \R$ is the {\bf curvature function of $K$}, and is merely the reciprocal of the Gauss curvature as a function of the outer unit normal. The class $C^2_+$ is used to denote those convex bodies whose support functions are $C^2$ and have positive radii of curvature (this is equivalent to definition from differential geometry, see \cite[Page 120]{Sh1}). From \cite[Theorem 2.7.1]{Sh1} every convex body can be approximated by convex bodies that are $C^2_+.$

	 Using the surface area measure allows us to state \textit{Cauchy's projection formula} \cite{Sh1}: for $\theta\in\mathbb{S}^{n-1}$ we have 
	\begin{equation*}
	   \vol_{n-1}\left(P_{\theta^{\perp}}K\right) =\frac{1}{2}\int_{\s^{n-1}}|\langle \theta,u \rangle| dS_K(u).
	\end{equation*}
We can use the above formula to write the support function of $\Pi K$ in the integral form:
\begin{equation}\label{e:proj_support}
    h_{\Pi K}(\theta)=\vol_{n-1}(P_{\theta^{\perp}}K) =\frac{1}{2}\int_{\s^{n-1}}|\langle \theta,u \rangle| dS_K(u).
\end{equation}
 
 \noindent We next have a generalization of surface area to measures with density, stated formally by Livshyts \cite{GAL}. \begin{definition}
	For a convex body $K\in\conbod$ and a Borel measure $\mu$ with density $\phi$ integrable on the boundary of $K,$ the $\mu$-surface area is defined implicitly by
\begin{equation}
    \label{eq:surface_mu}
    S_{\mu,K}(E)=\int_{n_K^{-1}(E)}\phi(x)d\mathcal{H}^{n-1}(x)
\end{equation}
for every Borel set $E \subset \s^{n-1}.$ If $K$ is of class $C^2_+$, then one has that $dS_{\mu,K}(u)=\phi\left(n_K^{-1}(u)\right)f_K(u)du.$
\end{definition}

\noindent We recall the definition of the isotropic measures on the sphere; this will be vital in Section~\ref{sec:iso}.
\begin{definition}
Let $\nu$ be a Borel measure on $\s^{n-1}.$ Then, we say $\nu$ is isotropic if, for every $\theta\in\s^{n-1},$ one has
$$\int_{\s^{n-1}}\langle \theta,u \rangle^2 d\nu(u)=\frac{\nu(\s^{n-1})}{n}.$$
\end{definition}
 
\noindent Recall that the support of a function $f$ is given by
	$\supp(f)=\overline{\{x:f(x)> 0\}}.$ A function $q:\Omega\to\R$ is said to be \textit{Lipschitz} on a domain (open, connected set with non-empty interior) $\Omega$ if, for every $x,y\in\Omega$, one has
	$|q(x)-q(y)|\leq L|x-y|$ for some $L>0$. The infimum over all such $L$ where this inequality is true is the Lipschitz norm of $q$ on $\Omega$, denoted $\|q\|_{\text{Lip}(\Omega)}$.Throughout this paper, when we say a function $f$ is Lipschitz on a domain $\Omega$, we implicitly assume that the domain is strictly in the interior of the support of $f$. Recall that one of the (equivalent) ways of defining a locally Lipschitz function on $\R^n$ is by saying a function $f$ is locally Lipschitz on a domain $\Omega \subset \text{supp}(f)$ if it is Lipschitz on compact subsets of $\Omega$. For example, given $K\in\conbod,$ the function $\chi_K$ is not Lipschitz on $\R^n$. However, $\chi_K$ is Lipschitz when restricted to its support, and thus we will still refer to this function as being locally Lipschitz. The Rademacher theorem will be critical for our investigations, and so we remind the reader of it here.
\begin{lemma}[Rademacher, \cite{Evans}]
Let $f$ be locally Lipschitz on a domain $\Omega$. Then, $f$ is differentiable almost everywhere in $\Omega$.
\label{l:rade}
\end{lemma}
	
	\noindent We now prove a lemma for estimating a concave function, where we say a function is concave if it is concave on its support.
	
\begin{lemma}
	\label{l:concave}
	Let $f$ be a concave function that is supported on a convex body $L\in\conbod_0$ such that, for some function $h(\theta)$,
	$$
\diff{f(r\theta)}{r}\bigg|_{r=0}\leq h(\theta) \leq 0, \quad  \text{for all }  \theta\in \s^{n-1},
$$ 
and $f(0) > 0.$ Define $z(\theta)=-h(\theta)^{-1}f(0),$ then
\begin{equation}\label{eq:concave_function} -\infty < f(r\theta)\leq f(0)\left[1-(z(\theta))^{-1}r\right]\end{equation}
whenever $\theta\in\s^{n-1}$ and $r\in[0,\rho_L(\theta)]$. In particular, if $f$ is non-negative, then we have
$$0 \leq f(r\theta)\leq f(0)\left[1-(z(\theta))^{-1}r\right] \quad  \mbox{and }  \rho_L(\theta)\leq z(\theta).
$$
One has $ f(r\theta)=f(0)\left[1-(z(\theta))^{-1}r\right]$ for $r\in[0,\rho_L(\theta)]$ if, and only if, $h(\theta)=\diff{f(r\theta)}{r}\bigg|_{r=0}$ and $\rho_L(\theta)=z(\theta).$
	\end{lemma}
	\begin{proof}
	From the concavity of the function $f$, one has
	$$f(r\theta)\leq f(0)\left[1+\diff{f(r\theta)}{r}\bigg|_{r=0}f(0)^{-1}r\right]\leq f(0)\left[1+h(\theta)f(0)^{-1}r\right],$$
	where the second inequality follows from the hypotheses. Then, the inequality \eqref{eq:concave_function} follows from the definition of $z(\theta)$. If $f$ is additionally non-negative, then one obtains that $\rho_L(\theta)\leq z(\theta)$ by using that $0\leq 1-(z(\theta))^{-1}r$ for a fixed $\theta\in\s^{n-1}$ and $r\in[0,\rho_L(\theta)],$ and then setting $r=\rho_L(\theta).$ For the equality conditions, suppose $f(r\theta)=f(0)\left[1-(z(\theta))^{-1}r\right];$ then, by definition of $L$ being the support of $f,$ one obtains $z(\theta)=\rho_L(\theta).$ From the definition of $z(\theta),$ this implies $h(\theta)=\diff{f(r\theta)}{r}\bigg|_{r=0}.$ Conversely, suppose $h(\theta)=\diff{f(r\theta)}{r}\bigg|_{r=0}$ and $\rho_L(\theta)=z(\theta).$ Then, from the concavity of $f$, one has $f(0)\left[1-(\rho_L(\theta))^{-1}r\right] \leq f(r\theta) \leq f(0)\left[1-(\rho_L(\theta))^{-1}r\right],$ and so there is equality.
	\end{proof}
   The next proposition shows how Lemma~\ref{l:concave} will be implemented in the paper. The proposition is a classical result \cite{GZA,Sh1,MA} on the properties of the covariogram of $K\in\conbod.$
\begin{proposition}
	For $K\in \conbod$, the associated covariogram $g_K$ is $1/n$-concave function on $DK$.   Furthermore, 
	\begin{equation*}
	    \diff{g_K(r\theta)}{r}\bigg|_{r=0}=-h_{\Pi K}(\theta), \quad \mbox{ for all } \theta \in \s^{n-1}.
	\end{equation*} In addition, for $\theta \in \s^{n-1}$ and $r \in [0,\rho_{DK}(\theta)],$ one has
	\begin{align*}
	    0\leq g_K(r\theta)\leq \vol_n(K)\left[1-\frac{h_{\Pi K}(\theta)}{n\vol_n(K)}r\right]^n \quad \text{and} \quad \rho_{DK}(\theta) \leq \rho_{n\vol_n(K)\Pi^\circ K}(\theta).
	\end{align*}
	Furthermore, $g_K(r\theta)^{1/n}=\vol^{1/n}_n(K)\left[1-\frac{h_{\Pi K}(\theta)}{n\vol_n(K)}r\right]$ if, and only if, $DK=n\vol_n(K)\Pi^\circ K.$
	\label{covario_prop}
	\end{proposition}
	\begin{proof}
	Using the fact that $$(1-t)[K\cap(K+x)]+t[K\cap(K+y)] \subset K \cap (K+(1-t)x+t y)$$ and the Brunn-Minkowski inequality, we have that $g_K^{1/n}(x)$ is concave for $x\in DK$.
	The next claim is shown in \cite{GZA}, and we will present a  proof from a more general case in Remark~\ref{re:old_fact} below.
	The last two claims follow from Lemma~\ref{l:concave}, with $f(x)=\left(\frac{g_K(x)}{\vol_n(K)}\right)^{1/n}$, $h(\theta)=\diff{f(r\theta)}{r}\bigg|_{r=0}=-\frac{h_{\Pi K}(\theta)}{n\vol_n(K)}$ and $L=DK$.
	\end{proof}
	
	\noindent We conclude this section with the following characterization of a simplex, which shows 
	
	\noindent $g_K(r\theta)^{1/n}=\vol^{1/n}_n(K)\left[1-\frac{h_{\Pi K}(\theta)}{n\vol_n(K)}r\right]$ for $\theta \in \s^{n-1}$ and $r \in [0,\rho_{DK}(\theta)]$ if, and only if, $K$ is a simplex.
	\begin{proposition}
	\label{p:simp}
	Let $K\in\conbod.$ The following are true if, and only if, $K$ is a simplex.
	\begin{enumerate}
	    \item \label{item:hom} Let $\theta\in\s^{n-1}$ and pick $r>0$ so that $K\cap (K+r\theta)\neq\emptyset.$ Then, $K\cap (K+r\theta)$ is homothetic to $K$ (\cite{RS1}, \cite[Section 6]{EGK64}, or \cite{Choquet}).
	    \item \label{item:aff} For every $\theta\in\s^{n-1}$, $g_{K}(r\theta)^{1/n}$ is an affine function in $r$ for $r\in[0,\rho_{DK}(\theta)]$ (Proposition~\ref{covario_prop} and Proposition~\ref{p:simp}, Item~\ref{item:hom}).
	    \item $DK=n\vol_n(K)\Pi^\circ K$ (Proposition~\ref{covario_prop} and Proposition~\ref{p:simp}, Item~\ref{item:aff}).
	\end{enumerate}
	\end{proposition}
	
	\section{Generalizations of Zhang's Inequality}
	\label{generalizations}
	\begin{theorem}[Zhang-type Inequality for General Measures]
	\label{t:weak_zhang}
	For $K\in\conbod$ and $\mu \in \Lambda$ we have
	\begin{equation*}
	    \mu_\lambda(K)\leq \mu\left(n\vol_n(K)\Pi^{\circ} K\right).
	\end{equation*}
Furthermore, this inequality is sharp over $\Lambda,$ in the sense that equality is obtained asymptotically for the measure $\gamma_n.$
	\end{theorem}
\begin{proof}	
The inequality was shown in $\eqref{eq:av}$ above; all that remains to show is the sharpness over $\Lambda$. Consider the case of  Gaussian measure $\gamma_n$ 
	and  $K=RB^n_2$, for $R>0$. 
	Then, $\Pi K=R^{n-1}\kappa_{n-1}B^{n}_2$ and $\Pi^\circ K=R^{1-n}\kappa_{n-1}^{-1}B^{n}_2.$ This yields $n\vol_n(K)\Pi^\circ K=nR\frac{\kappa_n}{\kappa_{n-1}}B^{n}_2.$ Thus, 
	$$\lim_{R\rightarrow\infty}\gamma_n(n\vol_{n}(RB^n_2)\Pi^\circ RB^n_2)=\gamma_n(\R^n)=1.$$
	Similarly, we can write, using the substitution $y=Rx$ and the evenness of $\gamma_n$:
	$$(\gamma_{n})_\lambda(RB^n_2):=\frac{1}{R^n\kappa_n}\int_{RB^n_2}\gamma_n(y-RB^n_2)dy=\frac{1}{\kappa_n}\int_{B^n_2}\gamma_n(R(B^n_2-x))dx.$$
	If $x$ is strictly in the interior of $B^n_2$, then one can find $r_x>0$ such that $r_xB^n_2 \subset B_2^n-x$; thus, $$1\geq\gamma_n(R(B^n_2-x))\geq\gamma_n(Rr_xB^n_2).$$
	But, the right-hand side goes to $1$ monotonically as $R\to\infty,$ 
	and so the sequence $\gamma_n(R(B^n_2-x))$ also monotonically increases to $1$ as $R\rightarrow \infty$. Hence, applying say dominated convergence yields
	$$\lim_{R\rightarrow\infty}(\gamma_{n})_\lambda(RB^n_2)=\frac{1}{\kappa_n}\int_{B^n_2}\gamma_n(\R^n)dx=1.$$
	\end{proof}
	We will now consider the case for $\nu\in\Lambda_{\text{rad}}$. This seems restrictive, but this requirement is necessary to obtain a generalization of Zhang's inequality that reduces to \eqref{e:Zhang_ineq} in the case of the Lebesgue measure. We start with a crucial result, Lemma~\ref{t:chak} below.

\begin{lemma}\label{t:chak}
Let $\nu\in\Lambda_{\text{rad}}$ have density $\varphi$, and let $f:\R^n\to\R^{+}$ be a compactly supported concave function such that $0\in \text{int}(\supp(f))$ and $f(0)=\max f(x).$
If $q \colon \R^{+} \to \R$ is an increasing function, then
\begin{equation*}
\int_{\supp(f)} q(f(x)) d\nu(x) \leq \beta \int_{\s^{n-1}}\int_0^{z(\theta)}\varphi(r\theta)r^{n-1}drd\theta,
\end{equation*}
where 
$$
z(\theta)=-\left(\diff{f(r\theta)}{r}\bigg|_{r=0}\right)^{-1}f(0) \quad \text{and } \quad \beta =n\int_0^1 q(f(0)t) (1-t)^{n-1} dt.$$
Equality occurs if, and only if, $z(\theta)=\rho_{supp(f)}(\theta)$ and, for every $\theta\in\s^{n-1},$ $r>0,$ $\varphi(r\theta)$ is independent of $r$ and $f(r\theta)$ is an affine function on its support.
\end{lemma}

\begin{proof} Set $L=\supp(f)$, which is a convex body with non-empty interior by the assumptions on $f$. Since $f$ is concave and $0\in \text{int}(\supp(f))$, we have that the directional derivative at $0$ exists, and, since $f$ obtains its maximum at $0$, we have
$$
\diff{f(r\theta)}{r}\bigg|_{r=0}\leq0, \quad  \text{for all }  \theta\in \s^{n-1}. 
$$ 
Hence, from Lemma~\ref{l:concave}, since $f$ is non-negative we have that
$$0 \leq f(r\theta)\leq f(0)\left[1-(z(\theta))^{-1}r\right] 
\quad \text{and} \quad  \rho_L(\theta)\leq z(\theta).$$
Integrating in polar coordinates, we may write the left-hand side as
\[
\int_{L} q(f(x)) d\nu(x) = \int_{\s^{n-1}}\int_0^{\rho_L(\theta)} q(f(r\theta)) \varphi(r\theta) r^{n-1} dr d\theta. 
\]
Hence, we have
\begin{equation}\label{e:ch1}
\int_{L} q(f(x)) d\nu(x) \leq \int_{\s^{n-1}}\int_0^{z(\theta)} q\left( f(0)\left[1-\frac{r}{z(\theta)}\right]\right) \varphi(r\theta) r^{n-1} dr d\theta. 
\end{equation}
Fix some direction $\theta \in \s^{n-1}$ and consider the function $\psi \colon (0,z(\theta)] \to \R^{+}$ given by, for $y\in (0,z(\theta)],$ 
\[
\psi(y)=\psi_\theta(y):= \beta \int_0^y \varphi(r) r^{n-1} dr - \int_0^y q\left(f(0)\left[1 - \frac{r}{y}\right]\right)\varphi(r) r^{n-1} dr,
\]
where $\beta > 0$ is a constant, independent of the direction $\theta$ and the measure $\nu$, chosen such that $\psi(y) \geq 0,$ and $\varphi(r):=\varphi(r\theta)$ for every $r>0$. Using the fact that $\varphi$ is continuous together with the fact that $q$ is integrable on each segment $(0,y] \subset (0,\infty)$, we may assert that $\psi(y) \to 0$ as $y \to 0^+$. Since $\psi$ is absolutely continuous on each $[a,b] \subset (0,y]$, $\psi$ may be represented by 
\[
\psi(y) = \psi(a) + \int_a^y \psi'(s) ds.
\]
Consequently, to have $\beta >0$ to be such that $\psi(y) \geq 0$, it suffices for $\beta$ to be selected so that $\psi'(y) \geq 0$ for almost every $y \in (0,z(\theta)]$. Differentiation of $\psi$ yields the representation
\[
\psi'(y) = \beta \varphi(y) y^{n-1} -q(0)\varphi(y)y^{n-1}-f(0) \int_0^y q'\left(f(0)\left[1 - \frac{r}{y}\right]\right)\frac{r^n}{y^2} \varphi(r) dr. 
\]
From the positivity of $\varphi(y)y^{n-1}$, we see that we must have
\[
\beta \geq q(0)+ f(0)\int_0^y q'\left(f(0)\left[1 - \frac{r}{y}\right]\right)\frac{r^n}{y^{n+1}} \frac{\varphi(r)}{\varphi(y)}dr,
\]
or equivalently, applying the change of variables $u=r/y$, we see that $\beta$ must satisfy
\[
\beta \geq q(0)+f(0) \int_0^1 q'(f(0)(1-u)) u^n \frac{\varphi(uy)}{\varphi(y)}du.
\]
Using that $\varphi$ radially non-decreasing implies $\varphi(uy)\leq \varphi(y)$ for $u\in[0,1]$, we see that setting $\beta$ to be the following satisfies our requirement:
\[
\beta = q(0)+f(0) \int_0^1 q'(f(0)(1-u)) u^n du = q(0)+f(0) \int_0^1 q'(f(0)t) (1-t)^n dt,
\]
where the second equality follows from the change of variables $t= 1-u$. Using integration by parts then yields
\[
\beta = n\int_0^1 q(f(0)t) (1-t)^{n-1} dt.
\]
Thus, \eqref{e:ch1} implies that 
\begin{align*}
\int_{L} q(f(x)) d\nu(x) &\leq  \beta \int_{\s^{n-1}} \int_0^{z(\theta)} \varphi(r\theta) r^{n-1} dr d\theta;
\end{align*}
we see that equality occurs when $z(\theta)=\rho_L(\theta)$ and $\varphi(uy)=\varphi(y)$ for all $u\in[0,1]$, which means $\varphi(u)$ is a constant. Furthermore, equality implies $f(r\theta)=f(0)\left(1-(z(\theta))^{-1}r\right),$ that is $f$ is affine on each ray. Thus, we have our claim.
\end{proof}
\begin{remark}
The result of Lemma~\ref{t:chak} is reminiscent of the following lemma by Berwald \cite{Berlem}:
Consider $K\in\conbod$, a concave function $f:K\to \R^+$ and a strictly increasing function $\psi:(0,\infty)\to\R$. Then, there exists a unique $\xi >0$ such that
$$\frac{1}{\vol_n(K)}\int_K\psi(f(x))dx=n\int_0^1\psi(\xi t)(1-t)^{n-1}dt,$$
and for every convex function $q$ such that the support of $q$ is an interval containing $\psi((0,\infty)),$ one has
$$\frac{1}{\vol_n(K)}\int_Kq(\psi(f(x)))dx\leq n\int_0^1q(\psi(\xi t))(1-t)^{n-1}dt.$$
The class of functions where equality occurs was established by Borell \cite{Borlem}.
\end{remark}

\begin{remark}
In the proof of Lemma~\ref{t:chak}, the following one-dimensional result was shown:
fix a positive integer $n$ and $\xi, s \in (0,\infty).$ Consider a continuous, monotonically increasing function $\phi$ on $(0,s)$, and let $q:\R^+\to\R$ be an increasing function. Then, for every $y\in (0,s),$ one has,
$$\beta \int_0^y \phi(r) r^{n-1} dr \geq \int_0^y q\left(\xi\left[1 - \frac{r}{y}\right]\right)\phi(r) r^{n-1} dr,$$
where
$$\beta =n\int_0^1 q(\xi t) (1-t)^{n-1} dt.$$
\end{remark}

\begin{theorem}[Zhang's Inequality for Radially Non-Decreasing Measures]
\label{th:zhnd}
Consider $\nu\in\Lambda_{\text{rad}}$ with density $\varphi$ and $K\in\conbod.$  Then,
\begin{equation*}
    {2n \choose n}\max\{\nu_\lambda(K),\nu_\lambda(-K)\}\leq\nu(n\vol_n(K)\Pi^\circ K),
\end{equation*}
with equality if, and only if, for every $\theta\in\s^{n-1},$ $\varphi(r\theta)$ is independent of $r,$ $r>0,$ and $K$ is a simplex.
\end{theorem}
\begin{proof}
In Lemma~\ref{t:chak}, set 
$
f(x)=\left(\frac{g_{K}(x)}{\vol_n(K)}\right)^{\frac{1}{n}}\!\!\!\!, \mbox{      }    q(x)=x^n, \mbox{ and } L=DK.$ Then, we have that $$\beta=n\int_0^1t^{n}(1-t)^{n-1}dt={2n \choose n}^{-1},$$
and from Proposition~\ref{covario_prop},
$$z(\theta)=\rho_{\Pi^\circ K}(\theta)n\vol_n(K)=\rho_{n\vol_n(K)\Pi^\circ K}(\theta).$$
Consequently,
$$\nu_\lambda(K)=\int_{DK}\frac{g_K(x)}{\vol_n(K)}d\nu(x)\leq {2n \choose n}^{-1}\int_{\s^{n-1}}\int_0^{\rho_{n\vol_n(K)\Pi^\circ K}(\theta)}\varphi(r\theta)r^{n-1}drd\theta.$$
Thus, the claim is established, since $\Pi K=\Pi (-K)$ and $\vol_n(K)=\vol_n(-K).$ Equality conditions are inherited from the equality conditions of Lemma~\ref{t:chak} and Proposition~\ref{p:simp}, as the latter shows $\left(\frac{g_{K}(x)}{\vol_n(K)}\right)^{\frac{1}{n}}$ is an affine function along rays restricted to $DK$ if, and only if, $K$ is a simplex.
\end{proof}


\section{Projection bodies Depending on Measures}\label{super_proj}
We begin this section by recalling a definition. A convex body $L$ is a \textit{zonoid} if, and only if, \cite[Chapter 4]{gardner_book} there exists an even, finite Borel measure $\nu$ on $\s^{n-1}$ such that the support function of $L$ is given by
$$h_L(\theta)=\int_{\s^{n-1}}|\langle\theta,u\rangle| d\nu(u).$$
For $K\in\conbod$, it follows from \eqref{e:proj_support} that $\Pi K$ is a zonoid. We now introduce another zonoid associated to $K,$ the measure dependent projection body of a convex body $K$.

\begin{definition}
The projection body of $K$ with respect to $\mu\in\Lambda$ is the convex body $\Pi_\mu K$ whose support function is given by
\begin{equation}\label{e:mu_polar_suppot}
h_{\Pi_{\mu}K}(\theta)= \frac{1}{2} \int_{\partial K} |\langle \theta, n_K(x) \rangle| \phi(x) dx=\frac{1}{2}\int_{\s^{n-1}}|\langle \theta, u\rangle|dS_{\mu,K}(u),
\end{equation}
where the last representation is due to the Gauss map. If $K$ is of class $C^2_+$ with curvature function $f_K$, then
\[
h_{\Pi_{\mu}K}(\theta) = \frac{1}{2}\int_{\s^{n-1}} |\langle u,\theta \rangle| \phi(n_K^{-1}(u)) f_K(u) du.
\]
\end{definition}
\noindent Using that $|\langle u,\theta\rangle|\leq 1$ for every $\theta\in\s^{n-1},$ we note that
$h_{\Pi_\mu K}(\theta)\leq \frac{\mu(\partial K)}{2},$ which yields $\Pi_\mu K\subset  \frac{\mu(\partial K)}{2}B_2^n.$

Next we would like to  highlight how these measure-dependent projection bodies differ from what has been studied before. We start with observing the behaviour of $\Pi_\mu K$ with respect to a linear transformation $T\in GL_n$.  We shall use the fact that, denoting $T^t$ to be the transpose of $T$, $T^{-1}$ to be the inverse of $T$, and $T^{-t}$ to be the inverse-transpose of $T$ allows one to obtain that $h_L(T^{-1}u)=h_{T^{-t}L}(u)$ for $L\in\conbod$. If we consider $\Pi$ and $\Pi_\mu$ as operators on $\conbod$, then it is well known \cite{Sh1} that linear transformations on the pre-image of $\Pi$ correspond to linear transformations on the image of $\Pi,$ that is $\Pi (TK) = |\det T| T^{-t}\left(\Pi K\right)$ for $K\in\conbod_0$ and $T\in GL_n$. In contrast, we will see that the operator $\Pi_\mu$ itself is changed by a linear transformation. 
\begin{proposition}
    Let $K\in\conbod_0$. Then, for every $T\in GL_n$, we have
    $$\Pi_\mu TK = |\det T| T^{-t}\left(\Pi_{\mu^T} K\right)$$
    where if $\mu$ has density $\phi$, then $\mu^T$ has density $\phi\circ T.$
    \label{lin_prop}
\end{proposition}
\begin{proof}
Begin by writing the support function of $\Pi_\mu TK$ explicitly:
$$
h_{\Pi_{\mu} T K}(\theta) =\frac{1}{2} \int_{\s^{n-1}}|\langle u,\theta\rangle| dS_{\mu,TK}(u)=\frac{1}{2} \int_{\s^{n-1}} h_{[-\theta, \theta]}(u) dS_{\mu,TK}(u)$$
where, in the second line, we used the fact that $|\langle u,\theta\rangle|$ is the support function of the line segment connecting $-\theta$ to $\theta$, denoted here as $[-\theta,\theta]$. Next, we can use the surjectivity of $T$ to find $z\in\R^n$ such that $Tz=\theta$ to write
$$h_{\Pi_{\mu} T K}(\theta)=\frac{1}{2} \int_{\s^{n-1}} h_{T([-z, z])}(u) dS_{\mu,TK}(u).$$
We recall from Chapter 5 of \cite{Sh1} the concept of mixed volumes. The mixed volume of a convex body $L$ and $(n-1)$ copies of a convex body $A$ is given by
$$V(A,L)=\frac{1}{n}\int_{\s^{n-1}}h_L(u)dS_A(u),$$
or, equivalently,
$$V(A,L)=\lim_{\epsilon\to 0^{+}}\frac{\vol_n(A+\epsilon L)-\vol_n(A)}{\epsilon}.$$
The concept of mixed volume was extended (e.g. in \cite{GAL,Naz}) to measures:
in our notation, let $\nu\in\Lambda$ with density $\varphi$ and suppose $A,L\subset\R^n$ are Borel measurable sets. Then, the $\nu$-mixed-measure of $A$ and $L$ is given by
\begin{equation}
\label{eq:mixed_vol_lim}
\nu(A,L)=\liminf_{\epsilon\to 0}\frac{\nu(A+\epsilon L)-\nu(A)}{\epsilon}.
\end{equation}
In \cite[Lemma 3.3]{GAL}, it was shown that, if $A$ and $L$ are convex bodies containing the origin, then
\begin{equation} \label{eq:mix_meas}
    \nu(A,L)=\int_{\s^{n-1}}h_L(u)dS_{\nu,A}(u).
\end{equation}
From \eqref{eq:mix_meas}, we see, with $\nu=\mu$, $A=TK$, and $L=[-\theta,\theta]$ that
$$\mu(TK,[-\theta,\theta])=2h_{\Pi_\mu TK}(\theta).$$
So, using \eqref{eq:mixed_vol_lim}, we can write
$$h_{\Pi_\mu TK}(\theta)=\frac{1}{2}\liminf_{\epsilon\to 0}\frac{1}{\epsilon}\left[\int_{T(K+\epsilon[-z, z])} \phi(x) d x-\int_{T K} \phi(x) d x\right].$$
Performing a variable substitution yields
$$h_{\Pi_\mu TK}(\theta)=|\det T|\frac{1}{2}\liminf_{\epsilon\to 0}\frac{1}{\epsilon}\left[\int_{K+\epsilon[-z, z]} \phi(Tx) d x-\int_{K} \phi(Tx) d x\right].$$
Next, using \eqref{eq:mix_meas}, this time with $\nu=\mu^T$, the measure whose density is $\phi(Tx)$, yields
$$h_{\Pi_\mu TK}(\theta)=|\det T|\frac{1}{2}\int_{\s^{n-1}}h_{[-z,z]}(u)dS_{\mu^T,K}(u).$$
Inserting that $z=T^{-1}\theta$ yields $$h_{\Pi_\mu TK}(\theta)=|\det T|\frac{1}{2}\int_{\s^{n-1}}|\langle T^{-1}\theta,u\rangle|dS_{\mu^T,K}(u)=|\det T|h_{\Pi_{\mu^T} K}(T^{-1}\theta).$$
Hence we have shown, for all $\theta\in\R^n$ and $T\in GL_n$, that
\begin{equation*}
    h_{\Pi_\mu TK}(\theta)=|\det T|h_{\Pi_{\mu^T} K}(T^{-1}\theta).
\end{equation*}
We then have $h_{\Pi_\mu TK}(\theta)=|\det T|h_{T^{-t}\left(\Pi_{\mu^T} K\right)}(\theta)$ and so our claim follows.
\end{proof}


Next, we will show that there does not exist a direct analogue of Petty's Projection inequality for $\Pi^\circ_\mu.$ The set inclusion $DK \subseteq n\vol_n(K) \Pi^\circ K$ from \eqref{e:set_inclusion} shows that the fundamental body involved in the inequalities of Petty and Zhang is not $\Pi^\circ K$ but $n\vol_n(K) \Pi^\circ K.$ With this in mind, one can write \eqref{e:Zhang_ineq} as
$$
{2n \choose n}\leq\frac{\vol_n(n\vol_n(K)\Pi^\circ K)}{\vol_n(K)}\leq \left(\frac{n\kappa_n}{\kappa_{n-1}}\right)^n.
$$
As we will see in Theorem~\ref{t:log} and Corollary~\ref{cor:s_zhang} below, this is essentially the correct quantity to consider for $\Pi^\circ_\mu K.$ Specifically, let $\mu\in\Lambda$ and $K\in\conbod.$ Then, in those results, the quantity
$$Pe(\mu,K)=\frac{\vol_n\left(\mu(K)\Pi^\circ_\mu K\right)}{\vol_n(K)}=\frac{\mu^{n}(K)\vol_n\left(\Pi^\circ_\mu K\right)}{\vol_n(K)}$$
is bounded from below. We will now show that this quantity does not permit an analogue of Petty's inequality for arbitrary $\mu\in\Lambda,$ that is it can be made unbounded from above. Indeed, consider a symmetric, convex body $K$ and a measure $\mu$ whose density is given by $\phi(x)=e^{-\|x\|_K}.$ For ease of computation, we will consider the simple case where $K$ is of class $C^2_+,$ e.g. one could set $K = B_2^n$.  For $t>0,$ consider the function $Pe(\mu, tK).$ Notice that $\mu$ is independent of $t.$ We will show that this function is unbounded as a function in $t$. Observe that
$$\lim_{t\rightarrow \infty}\mu(tK)=\lim_{t\rightarrow \infty}\int_{tK}e^{-\|x\|_K}dx=\int_{\R^n}e^{-\|x\|_K}dx=n!\vol_n(K).$$
We also see that
$$h_{\Pi_{\mu}tK}(\theta)=\frac{1}{2} \int_{S^{n-1}}|\langle u, \theta\rangle| e^{\left.-\| n^{-1}_{tK}(u)\right.\|_{K}} d S_{t K}(u)$$
and that $dS_{tK}(u)=t^{n-1}dS_K(u).$ Now, $n_{t K}^{-1}(u) \in \partial(t K)$ implies $\left\|n_{t K}^{-1}(u)\right\|_{K}=t$ and so we have
$$h_{\Pi_\mu tK}(\theta)=t^{n-1}e^{-t}h_{\Pi K}(\theta),$$
which implies
$\Pi_\mu (tK)=t^{n-1}e^{-t}\Pi K$ and $\Pi^\circ_\mu (tK)=t^{1-n}e^{t}\Pi^\circ K$. Thus, we obtain
$$Pe(\mu,tK)=t^{n-n^2}e^{nt}\vol_n(\Pi^\circ K)\frac{\mu^n(tK)}{\vol_n(K)}.$$
 Since $\mu^n(tK)/\vol_n(K) \to \left(n!\right)^n\vol^{n-1}_n(K) >0,$ one has that $Pe(\mu,tK)$ is unbounded as $t \rightarrow \infty.$

Our next goal is to provide an inequality for the volume of $\Pi_\mu K$. First, we introduce some notation. Recall, for $\mu\in\Lambda$ with density $\phi$ and $K\in\conbod$, we identify the $\mu$-measure of $\partial K$ via
\begin{equation*}
    \mu\left(\partial K\right):= \liminf_{\epsilon\to 0}\frac{\mu(K + \epsilon B_2^n)-\mu(K)}{\epsilon}=\int_{\s^{n-1}}dS_{\mu,K}(u)=\int_{\partial K}\phi(y)dy,
\end{equation*}
where the second equality is from \eqref{eq:mix_meas} and the last equality is from the Gauss map. Using Fubini's theorem on $\int_{\s^{n-1}}h_{\Pi_\mu K}(\theta)d\theta$, we have that
\begin{equation}\label{eq:mupartialsup}
    \mu\left(\partial K\right)=\frac{1}{\kappa_{n-1}}\int_{\s^{n-1}}h_{\Pi_\mu K}(\theta)d\theta.
\end{equation}
\begin{lemma}
\label{cor:sur}
    Let $K\in\conbod$ and $\mu\in\Lambda$, such that $\mu(\partial K)\neq 0.$ Then, we have the following:
    \begin{equation*}
    \left(\frac{n\kappa_n}{\kappa_{n-1}}\right)^n\kappa_n\leq \mu^n(\partial K)\vol_n(\Pi_\mu^\circ K).
\end{equation*}
In particular, when $\mu=\lambda$, we have
$$\left(\frac{n\kappa_n}{\kappa_{n-1}}\right)^n\kappa_n\leq \vol_{n-1}^n(\partial K)\vol_n(\Pi^\circ K).$$
\end{lemma}
\begin{proof}
From Jensen's inequality, we have that
$$\left(\int_{\s^{n-1}} \left(h_{\Pi_{\mu} K}(\theta)\right)^{-n} \frac{d \theta}{\vol_{n-1}\left(\s^{n-1}\right)}\right)^{-1 / n} \leq \int_{S^{n-1}} h_{\Pi_{\mu} K}(\theta) \frac{d \theta}{ \vol_{n-1}\left(\s^{n-1}\right)}.$$
Hence, we have
$$
\int_{\s^{n-1}} h_{\Pi_{\mu} K}(\theta)d\theta
\geq \vol_{n-1}\left(\s^{n-1}\right)^{\frac{n+1}{n}}\left(\int_{\s^{n-1}} h_{\Pi_{\mu} K}^{-n}(\theta)d\theta\right)^{-1 / n};
$$
\eqref{eq:mupartialsup} then yields
\begin{equation*}
    \kappa_{n-1}\mu(\partial K)\geq \vol_{n-1}\left(\s^{n-1}\right)^{\frac{n+1}{n}}\left(\int_{\s^{n-1}} h_{\Pi_{\mu} K}^{-n}(\theta)d\theta\right)^{-1 / n}.
\end{equation*}
Using \eqref{e:polar_volume} yields
\begin{equation*}
    \kappa_{n-1}\mu(\partial K)\geq \left(\frac{n\vol_n(\Pi_{\mu}^\circ K)}{\vol_{n-1}\left(\s^{n-1}\right)^{n+1}}\right)^{-1 / n}.
\end{equation*}
Taking the $n$th power of both sides 
$$\frac{\vol_{n-1}\left(\s^{n-1}\right)^{n+1}}{n\kappa^n_{n-1}}\leq \mu(\partial K)^{n}\vol_n(\Pi_{\mu}^\circ K).$$
Using that $\vol_{n-1}(\s^{n-1})=n\kappa_n$ yields the result.
\end{proof}

\noindent We remark that the case when $\mu=\lambda$ of Lemma~\ref{cor:sur} was shown in \cite{GP99}. We next show a full-dimensional analogue of this result.
\begin{theorem}
Consider some $K\in\conbod$ and $\mu\in\Lambda,$  with density $\phi$, such that $\mu(\partial K)<\infty.$ Then,
\begin{equation}\label{eq:almost}
    \left(\frac{n!\kappa_n}{\kappa_{n-1}}\right)^n\kappa_n\leq \left(\int_{\R^n}|\nabla e^{-\|z\|_K}|\phi\left(\frac{z}{\|z\|_K}\right)dz\right)^n\vol_n(\Pi_{\mu}^\circ K).
\end{equation}
\end{theorem}
\begin{proof}
Let $\phi$ be the density of $\mu$. From the definition of $\mu(\partial K)$ and $\Gamma(n)$, we may write
\begin{align*}
\Gamma(n)\mu(\partial K)&=\int_0^\infty\int_{\partial K}\phi(y)e^{-t}t^{n-1}dydt=\int_0^\infty\int_{\partial K}\phi\left(\frac{yt}{\|yt\|_K}\right)e^{-\|ty\|_K}t^{n-1}dydt
\\
&=\int_{\R^n}e^{-\|z\|_K}|\nabla \|z\|_K|\phi\left(\frac{z}{\|z\|_K}\right)dz,
\end{align*}
where the first equality follows from the fact that $\|y\|_K=1$ for $y\in\partial K$ and the second equality follows from, if $z\in \R^n$ has the polar decomposition $ty$, t$\in\R^+, y\in\partial K,$ then $|\nabla \|z\|_K|dz=t^{n-1}dydt.$
Notice that $|\nabla \|z\|_K|e^{-\|z\|_K}=|\nabla e^{-\|z\|_K}|.$
Thus, $$\Gamma(n)\mu(\partial K)=\int_{\R^n}|\nabla e^{-\|z\|_K}|\phi\left(\frac{z}{\|z\|_K}\right)dz,$$
which, by multiplying both sides of Lemma~\ref{cor:sur} by $\Gamma(n)^n=((n-1)!)^n,$ yields the result.
\end{proof}

\section{Lemmas Concerning The Intersection of a Convex Body with its Translates}
\label{sec:lem}
In this section, we will state some lemmas concerning $K\in\conbod$ and $K\cap (K+x)$, where $x\in DK$. These lemmas will be crucial for Section~\ref{sec:res}. We start our investigation by considering the convex body $\Pi_{\mu}K-\eta_{\mu,K}$ whose support function is given by:
\begin{equation*}
    h_{\Pi_{\mu}K-\eta_{\mu,K}}(\theta)=h_{\Pi_\mu K}(\theta)-\langle\eta_{\mu, K},\theta\rangle=\frac{1}{2}\int_{\partial K}|\langle\theta,n_K(y) \rangle| \phi(y) dy-\frac{1}{2}\int_K \langle\nabla\phi(y),\theta \rangle dy,
\end{equation*}
where $h_{\Pi_\mu K}(\theta)$ is given via \eqref{e:mu_polar_suppot} and the vector $\eta_{\mu,K}$ is defined as
\begin{equation}
     \eta_{\mu,K}=\frac{1}{2}\int_K \nabla \phi(y) d y,
     \label{e:eta}
 \end{equation}
or, by using the Gauss-Green-Ostrogradsky theorem,
\begin{equation*}
    \eta_{\mu,K}=\frac{1}{2}\int_{\partial K}n_K(x)\phi(x)dx.
\end{equation*}
We say $K$ is \textbf{$\mu$-projective} if $\eta_{\mu,K}=0$, in which case $h_{\Pi_{\mu}K-\eta_{\mu,K}}(\theta)=h_{\Pi_\mu K}(\theta)$. 

\begin{example}
\label{even_proj}
If $\lambda$ is the Lebesgue measure, then every $K\in\conbod$ is $\lambda$-projective. Furthermore, if $\mu$ is an even measure, then every symmetric $K\in\conbod$ is $\mu$-projective.
\end{example}
\begin{example}
Consider the case when $d\mu(x)=d\gamma_n(x)$ and $K$ such that $\int_K ye^{-|y|^2/2}dy=0,$ i.e. $K$ has a center of mass at the origin with respect to $\gamma_n.$ Then, notice that $\nabla e^{-|y|^2/2}=-ye^{-|y|^2/2}$ and so
 $$\int_K\langle \nabla e^{-|y|^2/2}, \theta\rangle dy=-\int_K\langle y, \theta\rangle e^{-|y|^2/2}dy=-\bigg\langle \int_K ye^{-|y|^2/2}dy,\theta \bigg\rangle=0,$$
 and hence we have that such $K$ is $\gamma_n$-projective and
 $$h_{\Pi_{\gamma_n} K-\eta_{\gamma_n,K}}(\theta)=h_{\Pi_{\gamma_n} K}(\theta)=\frac{(2\pi)^{-n/2}}{2}\int_{\partial K}\left|\langle\theta, {n}_{K}(y)\right\rangle| e^{-|x|^2/2} d y.$$
\end{example}

We are now able to state the following classification lemma, yielding an alternative representation of $h_{\Pi_\mu K}$.
\begin{lemma}\label{l:class}
For $K\in\conbod$ and $\mu\in\Lambda$ with density $\phi$, one has
\begin{equation*}
    \int_{\{z\in\partial K:\langle n_K(z),\theta\rangle \geq 0\}} |\langle n_K(z),\theta\rangle| \phi(z)dz=\left\langle\theta, \eta_{\mu,K}\right\rangle+h_{\Pi_{\mu} K}(\theta).
\end{equation*}
\end{lemma}
\begin{proof}
Begin by writing
$$A_1(\theta)=\int_{\{z\in\partial K:\langle n_K(z),\theta\rangle \geq 0\}} |\langle n_K(z),\theta\rangle| \phi(z)dz \quad \text{and} \quad A_2(\theta)=\int_{\{z\in\partial K:\langle n_K(z),\theta\rangle < 0\}} |\langle n_K(z),\theta\rangle| \phi(z)dz.$$
Then, using \eqref{e:mu_polar_suppot} shows that
$$
A_1(\theta)+A_2(\theta)=2 h_{\Pi_{\mu} K}(\theta) \quad \text{and} \quad A_1(\theta)-A_2(\theta)=2\left\langle\theta, \eta_{\mu,K}\right\rangle.
$$
Adding the two equations together and solving for $A_1(\theta)$ yields the result.
\end{proof}

The next lemma concerns the behaviour of $\partial K\cap (K+r\theta)$ as $r\to 0^+$. We recall that for $y\in \partial K,$ $n_K(y)$ may not be unique, thus  we set $\mathbf{n}_K(y)=\{u\in\s^{n-1}:u \text{ is an outer unit-normal of } \partial K \text{ at } y \}$. 
\begin{lemma}
\label{l:set}
Let $K\in\conbod$, $r>0$ and $\theta\in \s^{n-1}$. Then, 
$$\partial K\cap (K+r\theta) \subset \{y\in\partial K: \langle n_K(y),\theta\rangle\geq 0, \; \forall \; n_K(y)\in\mathbf{n}_K(y)\}.
$$  Moreover, $\partial K\cap (K+r\theta)$ converges to $\{y\in\partial K: \langle n_K(y),\theta\rangle\geq 0, \; \forall \; n_K(y)\in\mathbf{n}_K(y)\}$ as $r\to0^+$ with respect to set inclusion.
\end{lemma}
\begin{proof}
  Consider $\theta\in\s^{n-1}$ and fix  $r_0>0$ such that $r_0\theta \in DK$; then $\partial K\cap (K+r_0\theta)\neq \emptyset.$ Our claim follows if we can show that $y\in\partial K\cap (K+r\theta),$ for all $r\in (0,r_0]$ if, and only if, $\langle \theta, n_{K}(y) \rangle\geq 0.$ Begin by fixing $r\in (0,r_0]$. Now, if $y\in K+r\theta,$ then $y-r\theta\in K.$ But, if $y\in\partial K$, we must then have $\langle x, n_{K}(y) \rangle \leq \langle y, n_{K}(y) \rangle$ for all $x\in K$ and $n_K(y)\in\mathbf{n}_K(y)$; and, in particular, for $x=y-r\theta.$ This implies $0\leq \langle r\theta, n_{K}(y) \rangle$, and since $r>0$ we have $0\leq \langle \theta, n_{K}(y) \rangle.$ Thus, we have shown that $\partial K\cap (K+r\theta)\subset \{y\in\partial K: \langle n_K(y),\theta\rangle\geq 0\}$ for all $r$. Additionally, we note that $\partial K\cap (K+r\theta)$ is monotonically increasing with respect to set inclusion as $r$ decreases.
  
  Now, suppose $\partial K\cap (K+r\theta)$ increases to a proper subset of $\{y\in\partial K: \langle n_K(y),\theta\rangle\geq 0\}.$ Then, take $y$ in $\{y\in\partial K: \langle n_K(y),\theta\rangle\geq 0\}\setminus\cup_{r>0}[\partial K\cap (K+r\theta)].$ We have that $\langle n_K(y),\theta\rangle\geq 0$; which means $\langle y-r\theta, n_{K}(y) \rangle \leq \langle y, n_{K}(y) \rangle$ for all $0<r<r_0$. But from convexity, we then have $y-r\theta\in K$. Hence, $y\in K+r\theta$, which means $y\in\partial K\cap (K+r\theta)$, a contradiction. Thus $\partial K\cap (K+r\theta)\nearrow \{y\in\partial K: \langle n_K(y),\theta\rangle\geq 0\}$.
\end{proof}
We conclude this section by stating a lemma which shows what happens to the integral of a Lipschitz function over $(K+r\theta)\setminus K$ as $r\to 0^+.$
\begin{lemma}
\label{l:squeeze}
Fix some $K\in\conbod$ and $\theta\in\s^{n-1}$. Let $\Psi$ be a non-negative function which is Lipschitz on a domain $\Omega$ containing $\cup_{r\in [0,r_0]} (K+r\theta)$ for some small $r_0>0$. Furthermore, suppose $q:\R^+\to \R^+$ is a continuous function such that $\lim_{r\to 0^+} q(r) = 0$ and $|q(r)|\leq r.$ Then:
\begin{equation}
\label{eq:lip_int}
    \lim_{r\to 0^+}\frac{1}{r}\int_{(K+r\theta)\setminus K}\Psi(z+q(r)\theta)dz=\int_{\{y\in\partial K:\langle n_K(y),\theta\rangle \geq 0\}} |\langle n_K(y),\theta\rangle| \Psi(y)dy.
\end{equation}
\end{lemma}
\begin{proof}
We shall show our claim by bounding the quantity $$\frac{1}{r}\int_{(K+r\theta)\setminus K}\Psi(z+q(r)\theta)dz$$ above and below by functions that are continuous and equal the right-hand side of \eqref{eq:lip_int} as $r\to 0^+.$ Let us define two sets approximating $(K+r\theta)\setminus K$:
$$
\Omega_2(r, \theta) = \{y+t\theta: y \in\partial K\cap (K+r\theta), t\in (0,r]\} \mbox{ and }  \Omega_1(r, \theta) =\{y+t\theta: y\in\partial K,\langle n_K(y),\theta\rangle\geq 0, t\in (0,r]\}.
$$
We first show that
$$ \Omega_2(r, \theta) \setminus \partial K \subset (K+r\theta)\setminus K\subset \Omega_1(r, \theta).$$
The first inclusion is merely convexity. Indeed, if $K\cap (K+r\theta)=\emptyset,$ then the inclusion is trivial. Otherwise, consider $y \in \partial K\cap (K+r\theta)$. Since $y\in \partial K$ we get  $y+r\theta \in K+r\theta$; moreover, since $y \in K+r\theta$, we get, by convexity, that  $[y, y+r\theta] \subset K+r\theta$. Now let us show that $[y, y+r\theta]\cap \text{int} (K) =\emptyset.$ Suppose this assertion fails. Then there exists  $t_0>0$ such that $y+t_0\theta \in \text{int}(K).$ Thus 
 $\langle y+t_0\theta,n_K(y)\rangle\ < \langle y,n_K(y)\rangle.$ Using that $t_0>0$, we obtain $\langle n_K(y),\theta\rangle<0$. But, Lemma~\ref{l:set} reveals that $\partial K\cap (K+r\theta)\subset \{y\in\partial K: \langle n_K(y),\theta\rangle\geq 0\}$. Consequently, $\langle n_K(y),\theta\rangle<0$ implies $y\notin \partial K\cap (K+r\theta),$ a contradiction.

For the second inclusion, take $z\in (K+r\theta)\setminus K$. Then, there exists $x\in K$ such that $z=x+r\theta$. Now, consider the segment $[x,x+r\theta]$. We have that $x\in K$, but $x+r\theta \notin K$. Hence, there exists $y\in\partial K$ such that $y\in [x,x+r\theta]$. Furthermore, we can write $z=x+r\theta=y+t\theta$, where $t\in (0,r].$ We thus have that $\langle n_K(y),\theta\rangle\geq 0$, since we can write $x=y-(r-t)\theta$ and use that $\langle x, n_K(y)\rangle \leq \langle y, n_K(y)\rangle.$ Hence, we have the desired inclusions.

We note that $\Omega_2(r, \theta)\setminus\partial K$  differs from $\Omega_2(r, \theta)$ by a set of measure zero. Thus,
$$
\int_{\Omega_2(r, \theta)}\Psi(z+q(r)\theta)dz\leq\int_{(K+r\theta)\setminus K}\Psi(z+q(r)\theta)dz\leq \int_{\Omega_1(r, \theta)}\Psi(z+q(r)\theta)dz.
$$
Next, consider the coordinate transformation $z=(y,t)$, where $y\in\partial K$ and $t\in (0,r],$ and $z=y+t\theta.$ The Jacobian of this transformation is precisely the magnitude of the cosine of the angle between $n_K(y)$ and $\theta$, that is
$|\langle n_K(y),\theta\rangle|.$ Combining the above set inclusions and this coordinate transformation, we have
$$I_2(r,\theta)\leq\int_{(K+r\theta)\setminus K}\Psi(z+q(r)\theta)dz\leq I_1(r,\theta),$$
where $I_1(r,\theta)$ and $I_2(r,\theta)$ are continuous functions (in $r$) given by
\begin{align*}
    I_2(r,\theta)&=\int_{\partial K\cap (K+r\theta)}\int_{0}^r|\langle n_K(y),\theta\rangle|\Psi(y+t\theta+q(r)\theta)dt dy,
    \\
    I_1(r,\theta)&=\int_{\{y\in\partial K: \langle n_K(y),\theta\rangle\geq 0\}}\int_{0}^r|\langle n_K(y),\theta\rangle|\Psi(y+t\theta+q(r)\theta)dt dy.
\end{align*}
 Since $\Psi$ is Lipschitz, it is bounded on $K$, and so we can use dominated convergence and the Lebesgue differentiation theorem \cite{Stein}, to obtain
\begin{align*}
\lim_{r\to 0^+}\frac{1}{r}I_1(r,\theta)&=\int_{\{y\in\partial K: \langle n_K(y),\theta\rangle\geq 0\}}\lim_{r\to 0^+}\frac{1}{r}\left(\int_{0}^r|\langle n_K(y),\theta\rangle|\Psi(y+t\theta+q(r)\theta)dt\right)dy
\\
&=\int_{\{y\in\partial K: \langle n_K(y),\theta\rangle\geq 0\}}\lim_{r\to 0^+}\frac{1}{r}\left(\int_{0}^r|\langle n_K(y),\theta\rangle|\Psi(y+t\theta)dt\right)dy
\\
&=\int_{\{y\in\partial K:\langle n_K(y),\theta\rangle \geq 0\}} |\langle n_K(y),\theta\rangle| \Psi(y)dy,
\end{align*}
where we used the fact that $\frac{1}{r}\int_{0}^r|\Psi(y+t\theta)-\Psi(y+t\theta+q(r)\theta)|dy\leq \|\Psi\|_{\text{Lip}(\Omega)} |q(r)|$ goes to zero as $r\to0^+$ by hypothesis when going from the first to second line. Next, we expand the set being considered for $I_2(r,\theta)$ by adding and subtracting an integration over $\{y\in\partial K: \langle n_K(y),\theta\rangle\geq 0\}\setminus (K+r\theta)$ to obtain:
\begin{align*}
I_2(r,\theta)=I_1(r,\theta)-\int_{\{y\in\partial K: \langle n_K(y),\theta\rangle\geq 0\}\setminus (K+r\theta)}\int_{0}^r|\langle n_K(y),\theta\rangle|\Psi(y+t\theta+q(r)\theta)dt dy=I_1(r,\theta)-I_3(r,\theta).
\end{align*}
Thus, we see that
$$\lim_{r\to 0^+}\frac{1}{r}I_2(r,\theta)=\lim_{r\to 0^+}\frac{1}{r}I_1(r,\theta)-\lim_{r\to 0^+}\frac{1}{r}I_3(r,\theta).$$
We now show that $$\lim_{r\to 0^+}\frac{1}{r}I_3(r,\theta)=0.$$
Since $\Psi$ is Lipschitz in $K\cup (K+r\theta)$, it is bounded on $K+r\theta$, by some $L>0$; from Cauchy-Schwarz, one also has $|\langle n_K(y),\theta \rangle|\leq 1$ for all $\theta\in \s^{n-1}$; so, we compute
\begin{align*}\bigg|\frac{1}{r}I_3(r,\theta)\bigg|&\leq\int_{\{y\in\partial K: \langle n_K(y),\theta\rangle\geq 0\}\setminus (K+r\theta)}\left(\frac{1}{r}\int_0^r\big|\Psi(y+t\theta+q(r)\theta)\langle n_K(y),\theta \rangle \big|dt\right)dy 
\\
&\leq L\vol_{n-1}\left(\{y\in\partial K: \langle n_K(y),\theta\rangle\geq 0\}\setminus (K+r\theta)\right),
\end{align*}
and this goes to zero as $r\to 0^{+}$ via Lemma~\ref{l:set}.
\end{proof}

\section{Measure-Dependent Covariogram for Convex Bodies}
\label{sec:res}
\begin{definition}
Let $\mu\in\Lambda$ with density $\phi$ and consider $K\in\conbod$. The $\mu$-covariogram of $K$ is given as
\begin{equation}\label{e:mu_covario}
    g_{\mu,K}(x)=\mu(K\cap(K+x))=(\chi_K\phi \star \chi_{-K})(x).
\end{equation} \end{definition}
In \eqref{e:mu_average}, we saw the average of $\mu$ with respect to $\lambda$ was related to integrating the covariogram. In the same spirit, we can define the average of $\lambda$ with respect to $\mu$ as
$\lambda_\mu(f)=\frac{1}{\mu(K)}\int_{\R^n}g_{\mu,K}(x)dx.$ The following, however, shows this is unnecessary.
\begin{example}
\label{int_g_mu}
For $K\in\conbod$ and $\mu\in\Lambda$ with density $\phi$, we obtain:
$$\mu(K)\lambda_\mu(K)=\int_{\R^n}g_{\mu,K}(x)dx=\int_K\vol_n(y-K)\mu(y)=\vol_n(K)\mu(K),$$
since $\chi_{(K+x)}(y)=\chi_{(y-K)}(x)$ and $\vol_n(y-K)=\vol_n(K).$
\end{example}
We  define the brightness of a convex body $K$ with respect to $\mu$, or the $\mu$-brightness of K to be the derivative in the radial direction of $g_{\mu,K}(r\theta)$ evaluated at $0$. We shall now determine the $\mu$-brightness of $K$, the proof of which is inspired by the proof of the derivative of $g_K$, found in \cite{MA}.
\begin{theorem}
Let $K\in\conbod$. Suppose $\Omega$ is a domain containing $K$, and consider $\mu\in\Lambda$ with density $\phi$ locally Lipschitz on $\Omega$. Then, for every $\theta\in\s^{n-1},$ the brightness of $K$ with respect to $\mu$ is -$h_{\Pi_\mu K-\eta_{\mu,K}}(\theta)$ i.e.
\begin{equation}\label{e:deriv_g_mu_covario}
    \diff{g_{\mu,K}(r\theta)}{r}\bigg|_{r=0}=-h_{\Pi_{\mu}K-\eta_{\mu,K}}(\theta), \quad \theta\in\s^{n-1}.
    \end{equation}
\end{theorem}
\begin{proof}
We must show that the quantity
\begin{equation*}
    \diff{g_{\mu,K}(r\theta)}{r}\bigg|_{r=0}=\lim_{r\to 0^+}\frac{1}{r}\left(\int_K\chi_K(y-r\theta)\phi(y)dy-\int_K\phi(y)dy\right)
\end{equation*}
equals the right-hand side of \eqref{e:deriv_g_mu_covario}. First, observe that we can equivalently write this as
$$
    \diff{g_{\mu,K}(r\theta)}{r}\bigg|_{r=0}=\lim_{r\to 0^+}\frac{1}{r}\left(\int_{K\cap( K-r\theta)}\phi(y+r\theta)dy-\int_K\phi(y)dy\right)
$$
by performing the variable substitution $y\to y+r\theta.$ Next, since $\phi$ is locally Lipschitz in $\Omega$, it is Lipschitz on $K$; we have, via Lemma~\ref{l:rade}, that $\phi$ is differentiable almost everywhere in $K$. In particular, $\phi$ is Lipschitz radially and hence its radial derivative is bounded almost everywhere (from the classical fact that one-dimensional Lipschitz functions have a derivative bounded almost everywhere \cite{Stein}). So, using dominated convergence, we may write
$$\lim_{r\to 0^+} \frac{1}{r}\left(\int_K \phi (y+r\theta)dy-\int_K \phi(y)dy\right)=\int_K \langle \nabla \phi(y),\theta\rangle dy.$$

Hence, we may write, recalling the definition of $\eta_{\mu,K}$ in \eqref{e:eta},
\begin{align*}
    \diff{g_{\mu,K}(r\theta)}{r}\bigg|_{r=0}
    &=\int_K \langle \nabla \phi(y),\theta\rangle dy+\lim_{r\to 0^+}\frac{1}{r}\left(\int_{K\cap( K-r\theta)}\phi(y+r\theta)dy-\int_K\phi(y+r\theta)dy\right)
    \\
    &=\int_K \langle \nabla \phi(y),\theta\rangle dy+\lim_{r\to 0^+}\frac{1}{r}\left(\int_{K\cap( K+r\theta)}\phi(y)dy-\int_{K+r\theta}\phi(y)dy\right)
    \\
    &=\int_K \langle \nabla \phi(y),\theta\rangle dy-\lim_{r\to 0^+}\frac{1}{r}\left(\int_{K+r\theta}\phi(y)dy-\int_{K\cap( K+r\theta)}\phi(y)dy\right)
    \\
    &=\int_K \langle \nabla \phi(y),\theta\rangle dy-\lim_{r\to 0^+}\frac{1}{r}\int_{(K+r\theta)\setminus K}\phi(y)dy.
    \\
    &=2\langle\eta_{\mu,K},\theta\rangle-\lim_{r\to 0^+}\frac{1}{r}\int_{(K+r\theta)\setminus K}\phi(y)dy.
\end{align*}
 
From Lemma~\ref{l:squeeze}, with $q(r)=0$ and $\Psi=\phi,$ we have 
$$\lim_{r\to 0^+}\frac{1}{r} \int_{(K+r\theta)\setminus K}\phi(y)dy=\int_{\{s\in\partial K:\langle n_K(s),\theta\rangle \geq 0\}} |\langle n_K(s),\theta\rangle| \phi(s)ds=\left\langle\theta, \eta_{\mu,K}\right\rangle+h_{\Pi_{\mu} K}(\theta),$$
where the last equality follows from Lemma~\ref{l:class}. Inserting into the above, we obtain
$$\diff{g_{\mu,K}(r\theta)}{r}\bigg|_{r=0}=2\left\langle\theta, \eta_{\mu,K}\right\rangle-\lim_{r\to 0}\int_{(K+r\theta)\setminus K}\phi(y)dy=\left\langle\theta, \eta_{\mu,K}\right\rangle-h_{\Pi_{\mu} K}(\theta).$$
\end{proof}

\begin{remark}
\label{re:old_fact}
Taking $\mu=\lambda$, that is $\phi\equiv1$, in \eqref{e:deriv_g_mu_covario} yields
$$
	    \diff{g_K(r\theta)}{r}\bigg|_{r=0}=-h_{\Pi K}(\theta),
$$
which finishes the proof of Proposition~\ref{covario_prop}.
\end{remark}
\begin{remark}
Since $g_{\mu,K}$ is decreasing as a function of $r$, we have the derivative is non-positive. Thus, $h_{{\Pi_\mu K} -\eta_{\mu,K}}(\theta) \geq 0$ for every $\theta\in\s^{n-1}$. From convexity, this means that $0\in \Pi_{\mu} K-\eta_{\mu,K}.$ So, the vector $\eta_{\mu,K}$ does not shift $\Pi_{\mu} K$ completely away from the origin.
\end{remark}

\noindent We now show that the $\mu$-covariogram inherits the concavity of $\mu$.
\begin{lemma}
\label{l:covario_concave}
Consider a class of convex bodies $\mathcal{C}\subseteq\conbod$ with the property that $K\in \mathcal{C} \rightarrow K\cap(K+x)\in\mathcal{C}$ for every $x\in DK$. Let $\mu$ be a Borel measure finite on every $K\in\mathcal{C}.$ Suppose $F$ is a continuous and invertible function such that $\mu$ is $F$-concave on $\mathcal{C}$. Then, for $K\in\mathcal{C},$ $g_{\mu,K}$ is also $F$-concave, in the sense that, if $F$ is increasing, then $F\circ g_{\mu,K}$ is concave, and if $F$ is decreasing, then $F\circ g_{\mu,K}$ is convex.
\end{lemma}
\begin{proof}
We first observe the following set inclusion:
for $x, y \in \R^{n}$ and $\lambda \in [0,1]$, we have from convexity that
$$
\begin{aligned}
K \cap(K+(1-\lambda) x+\lambda y) &=K \cap((1-\lambda)(K+x)+\lambda(K+y)) \\
& \supset(1-\lambda)(K \cap(K+x))+\lambda(K \cap(K+y)).
\end{aligned}
$$
Using this set inclusion, we obtain that
$$g_{\mu,K}((1-\lambda) x+\lambda y) \geq \mu((1-\lambda)(K \cap(K+x))+\lambda(K \cap(K+y))).$$
From the fact that $\mu$ is $F$-concave, we obtain
$$
\begin{aligned}
g_{\mu,K}((1-\lambda) x+\lambda y) & \geq F^{-1}\left((1-\lambda) F\left(\mu(K \cap(K+x))\right)+\lambda F\left(\mu(K \cap(K+y))\right)\right) \\
&=F^{-1}\left((1-\lambda) F(g_{\mu,K}(x))+\lambda F(g_{\mu,K}(y))\right).
\end{aligned}
$$
\end{proof}

We now show that the measure-dependent projection bodies satisfy an analogue of \eqref{e:set_inclusion}. Let $f=F\circ g_{\mu,K}$ for some $K\in\conbod$ and $\mu\in\Lambda$ such that $\mu$ is $F$-concave, $F$ a non-negative, differentiable, increasing function. Then, from Lemma~\ref{l:covario_concave}, $f$ is a concave function supported on $DK$. We set $z(\theta)=\frac{F(\mu(K))}{F^\prime(\mu(K))}\rho_{\left({\Pi_{\mu}K-\eta_{\mu,K}}\right)^\circ}(\theta)$ and obtain, via Lemma~\ref{l:concave} that, $\rho_{DK}(\theta)\leq z(\theta).$ Hence,
\begin{equation}\label{e:mu_set_incl}
DK\subseteq \frac{F(\mu(K))}{F^\prime(\mu(K))} \left({\Pi_{\mu}K-\eta_{\mu,K}}\right)^\circ.
\end{equation}
To make this result more clear, suppose that $K$ is $\mu$-projective. Then, $\Pi_{\mu}K-\eta_{\mu,K} = \Pi_{\mu}K$
and so
$$DK\subseteq \frac{F(\mu(K))}{F^\prime(\mu(K))}\Pi^\circ_{\mu}K.$$
When $\mu$ is the full dimensional Lebesgue measure, we can set $F(x)=x^{1/n}$ via the Brunn-Minkowksi inequality and obtain the extremes of \eqref{e:set_inclusion}:
$DK\subseteq n\vol_n(K)\Pi^\circ K.$

Now, for $\mu\in\Lambda$ with density $\phi$ and $K \in \mathcal{K}^n$, we have, for every $x\in\R^n,$
$
\left(\inf_{x\in \partial K}\phi(x)\right) h_{\Pi K}(x)\leq h_{\Pi_{\mu} K}(x),$
or equivalently, 
$$\inf_{x\in \partial K}\phi(x){\Pi^\circ_\mu K}\subseteq{\Pi^\circ K},$$
with equality if, and only if, $\phi$ is a constant on $\partial K$. Using that $\vol_n(K)\Pi^\circ K \subseteq DK$ from \eqref{eq:spectral} and \eqref{e:mu_set_incl} yields,
\begin{equation}
    \vol_n(K)\left(\inf_{x\in \partial K}\phi(x)\right){\Pi^\circ_\mu K}\subseteq\vol_n(K){\Pi^\circ K}\subseteq  DK \subseteq\frac{F(\mu(K))}{F^\prime(\mu(K))} \left({\Pi_{\mu}K-\eta_{\mu,K}}\right)^\circ.
    \label{eq:big_set}
\end{equation}

\noindent There is equality in the first set-inclusion if, and only if, $\phi$ is a constant on the boundary of $K.$ From the symmetry of $DK$ and $\Pi^\circ_\mu K,$ in conjunction with the equality conditions of Lemma~\ref{l:concave}, there is equality in the third set inclusion if, and only if, $K$ is $\mu$-projective and $F\circ g_{\mu,K}(r\theta)=F(\mu(K))\left[1-\frac{r}{\rho_{DK}(\theta)}\right]$ for $r\in [0,\rho_{DK}(\theta)].$

We conclude this section by discussing the Gaussian measure. In \cite{GZ}, Gardner and Zvavitch showed inequalities concerning the concavity of the Gaussian measure, by considering alternatives to Minkowski addition. They then asked if the following is true: for $K,L\in\conbod_0$ and $t\in[0,1]$,
\begin{equation}\label{e:gamma_gaussian}
    \gamma_n\left((1-t) K + t L\right)^{1/n}\geq (1-t)\gamma_n(K)^{1/n} + t \gamma_n(L)^{1/n},
\end{equation}
i.e. is the $\gamma_n$ measure $1/n$-concave over $\conbod_0$? A counterexample was shown in \cite{PT} when $K$ and $L$ are not symmetric. Important progress was made in \cite{KL}, which lead to the resolution of inequality \eqref{e:gamma_gaussian} in \cite{EM} for symmetric convex bodies. The question remains open for convex bodies whose Gaussian center of mass is at the origin. 

We see that \eqref{e:gamma_gaussian} improves the concavity of $\gamma_n$ on symmetric convex bodies. We would like to use this fact, however \eqref{e:gamma_gaussian} cannot be applied directly to study $g_{\gamma_n,K}$, because, in general, $K$ being symmetric does not imply the set $K\cap(K+x)$ is symmetric. Thus, we would like to create a symmetric version of the covariogram. Let us now deduce this covariogram. One can show that if $K$ is symmetric, then $K\cap(K+x) -x/2$ is symmetric. So, for $x,y\in DK$, we can apply  the inequality \eqref{e:gamma_gaussian} to the symmetric convex sets $K\cap(K+y)-y/2$ and $K\cap(K+x)-x/2$ and obtain, for $t\in[0,1],$
\begin{align*}\gamma_n\left((1-t) (K\cap(K+x)) + t(K\cap(K+y)) -\left((1-t)\frac{x}{2}+t \frac{y}{2}\right)\right)^{1/n}
\\
\geq (1-t)\gamma_n(K\cap(K+x)-x/2)^{1/n} + t \gamma_n(K\cap(K+y)-y/2)^{1/n},
\end{align*}
i.e. the function $$x\to\gamma_n(K\cap(K+x)-x/2)^{1/n}$$ is concave. It is not hard to show, by using that both $K$ and $K\cap(K+x)-x/2$ are symmetric, that $$K\cap(K+x)-x/2=(K-x/2)\cap(K+x/2).$$ From these considerations, we are motivated to define the following polarized covariogram.
\begin{definition}
\label{def:polar}
Consider $K\in\conbod$ and a Borel measure $\mu$ on $\R^n$. The polarized $\mu$-covariogram is defined as
$$r_{\mu,K}(x)=\mu((K-x/2)\cap(K+x/2))$$
and the polarized $\mu$-brightness of $K$ is given as
$$\diff{r_{\mu,K}(r\theta)}{r}\bigg|_{r=0}.$$
\end{definition}
\noindent Furthermore, we note that if $R$ is a continuous and invertible function such that $\mu$ is $R$-concave on symmetric sets, then $R\circ r_{\mu,K}$ is concave or convex (merely repeat the computation done with $\gamma_n^{1/n}$ preceding Definition~\ref{def:polar}). 

\noindent In the volume case, it is clear from the translation invariance of the Lebesgue measure that $r_{\lambda,K}(x)=g_{K}(x)$, and this property is in general not true for other measures. In the next proposition we show that the polarized $\mu$-brightness of $K$ coincides with the $\mu$-brightness of $K$ under the assumptions that $K$ is symmetric and $\mu$ is even. 
\begin{proposition}
Consider a symmetric $K\in\conbod_0$, and let $\mu\in\Lambda$ be an even measure with density $\phi$ that is Lipschitz on a domain $\Omega$ containing $K$. Then,
\begin{equation*}
    \diff{r_{\mu,K}(r\theta)}{r}\bigg|_{r=0}=-h_{\Pi_\mu K}(\theta).
    \end{equation*}
\end{proposition}

\begin{proof}
We must show that the quantity
\begin{equation*}
    \diff{r_{\mu,K}(r\theta)}{r}\bigg|_{r=0}=\lim_{r\to 0^+}\frac{1}{r}\left(\int_{(K-\frac{r\theta}{2})\cap (K+\frac{r\theta}{2})}\phi\left(y\right)dy-\int_K\phi(y)dy\right).
\end{equation*}
equals $-h_{\Pi_\mu K}(\theta)$. Since $\phi$ is locally Lipschitz in $\Omega$, we can again use dominated convergence to write
$$\lim_{r\to 0^+} \frac{1}{r}\left(\int_K \phi \left(y+\frac{r\theta}{2}\right)dy-\int_K \phi(y)dy\right)=\frac{1}{2}\int_K \langle \nabla \phi(y),\theta\rangle dy=\langle\eta_{\mu,K},\theta\rangle=\langle0,\theta\rangle=0,$$
since $K$ is symmetric and $\mu$ is even. Hence, using the variable substitution $y\to y-\frac{r\theta}{2}$ in the first integral,
\begin{align*}
    &\diff{r_{\mu,K}(r\theta)}{r}\bigg|_{r=0}
    =\lim_{r\to 0^+}\frac{1}{r}\left(\int_{K\cap (K+r\theta)}\phi\left(y-\frac{r\theta}{2}\right)dy-\int_{K}\phi\left(y+\frac{r\theta}{2}\right)dy\right)
    \\
    &=\lim_{r\to 0^+}\frac{1}{r}\left(\int_{K\cap( K+r\theta)}\phi\left(y-\frac{r\theta}{2}\right)dy-\int_{K+r\theta}\phi\left(y-\frac{r\theta}{2}\right)dy\right)
    \\
    &=-\lim_{r\to 0^+}\frac{1}{r}\int_{(K+r\theta)\setminus K}\phi\left(y-\frac{r\theta}{2}\right)dy,
\end{align*}
where, in the second equality, we used the substitution $y\to y-r\theta$ on the second integral. We note that
$$\frac{1}{r}\left|\int_{(K+r\theta)\setminus K}\left(\phi\left(y-\frac{r\theta}{2}\right)-\phi(y)\right)dy\right| \leq \frac{\|\phi\|_{\text{Lip}(\Omega)}}{2}\int_{(K+r\theta)\setminus K}dy,$$
which goes to zero as $r\to 0^{+}.$ Therefore, we obtain
$$\diff{r_{\mu,K}(r\theta)}{r}\bigg|_{r=0}=-\lim_{r\to 0^+}\frac{1}{r}\int_{(K+r\theta)\setminus K}\phi\left(y\right)dy.$$

\noindent Next, we invoke Lemma~\ref{l:squeeze} with $\phi=\Psi$, $q(r)=0$ to obtain, since $K$ is symmetric and $\phi$ is even,
\begin{align*}\lim_{r\to 0^+}\frac{1}{r}\int_{(K+r\theta)\setminus K}\phi(y)dy=\int_{\{s\in\partial K:\langle n_K(s),\theta\rangle \geq 0\}} |\langle n_K(s),\theta\rangle| \phi(s)ds=h_{\Pi_{\mu} K}(\theta)
\end{align*}
from Lemma~\ref{l:class}. Consequently, we obtain
$$\diff{r_{\mu,K}(r\theta)}{r}\bigg|_{r=0}=-\lim_{r\to 0^+}\frac{1}{r}\int_{(K+r\theta)\setminus K}\phi(y)dy=-h_{\Pi_{\mu} K}(\theta).$$
\end{proof}

The function $r_{\mu,K}$ still makes sense for arbitrary $K\in\conbod_0.$ In fact, the support of $r_{\mu,K}$ is $DK.$ Indeed, we will show that, for fixed $\lambda \in [0,1]$ and $K,L\in\conbod_0,$ one has $$S:=\{x\in\R^n: (K-\lambda x)\cap(L+(1-\lambda)x)\neq \emptyset\}=K+(-L).$$ Setting $\lambda =1/2$ and $L=K$ will then yield that the support of $r_{\mu,K}$ is $DK.$ So, take $x\in S.$ Then, there exists $y\in K$ and $z\in L$ such that there is some $a\in (K-\lambda x)\cap(L+(1-\lambda)x)$ with the representations $z+(1-\lambda)x=a=y-\lambda x.$
Solving for $x,$ we obtain $x=y-z=y+(-z)\in K+(-L).$ Conversely, let $x\in K+(-L).$ Then, there exists $y\in K$ and $z\in L$ so that $x=y-z.$ Define $a_1=y-\lambda x$ and $a_2=z+(1-\lambda)x.$ If we show that $a_1=a_2 \equiv a,$ then $a\in (K-\lambda x)\cap(L+(1-\lambda)x)$ and so $x\in S.$ But, notice that $a_2-a_1=z-y+x=0$ from the representation of $x,$ and so the claim follows.

Let $f=R\circ r_{\mu,K}$ for some symmetric $K\in\conbod_0$ and even $\mu\in\Lambda$ such that $\mu$ is $R$-concave, $R$ a non-negative, differentiable, increasing function. Then, $f$ is a concave function supported on $DK$. We set $z(\theta)=\frac{R(\mu(K))}{R^\prime(\mu(K))}\rho_{{\Pi_{\mu}^\circ K}}(\theta)$ and obtain, via Lemma~\ref{l:concave} that, $$\rho_{DK}(\theta)\leq z(\theta).$$ Hence,
\begin{equation}\label{e:polar_mu_set_incl}
DK\subseteq \frac{R(\mu(K))}{R^\prime(\mu(K))} \Pi_{\mu}^\circ K.
\end{equation}
As an example, if $\mu=\gamma_n,$ then from the resolution of \eqref{e:gamma_gaussian} for symmetric convex bodies, we can set $R(x)=x^{1/n}$ and obtain
$$DK\subseteq n\gamma_n(K) \Pi_{\gamma_n}^\circ K.$$

\section{Measure-Dependent Projection Bodies for Functions}
\label{sec:functions}
We are motivated to generalize measure-dependent projection bodies $\Pi_\mu K$ to a larger class of functions. To do so, we introduce the following definition.
\begin{definition}
For $\mu\in\Lambda$ and a convex body $K\in\conbod$, consider a non-negative $f\in L^1(\mu,\partial K)$. Then the projection body of $f$ with respect to $\mu$ and $K$ is the convex body whose support function is given by
\begin{equation}\label{e:weak_supp_f_mu}
    h_{\Pi_{\mu,K} f}(\theta)=\frac{1}{2}\int_{\partial K}|\langle\theta,n_K(y) \rangle| f(y)\phi(y) dy.
\end{equation}
\end{definition}

We see that \eqref{e:mu_polar_suppot} is a special case of \eqref{e:weak_supp_f_mu}, as one has that $h_{\Pi_{\mu,K} \chi_K}=h_{\Pi_\mu K}$. The advantage of this definition is that we can allow $f$ and $\phi$ to have different restrictions. Ideologically, $f$ is a function that only need to be {\it nice} on $K$, while $\phi$ needs to be {\it smooth} beyond $K$ (where {\it nice} and {\it smooth} will be stated precisely in context). We next define the covariogram of a function to be the following:
\begin{definition}
Fix $K\in\conbod$. Consider $\mu\in\Lambda$ with density $\phi$. Furthermore, let $f\in L^1(\mu, K)$. Then, we define the $\mu$-covariogram of $f$ as
\begin{equation}
    g_{\mu,f}(K,x)=\int_{K\cap(K+x)} f(y-x)\phi(y)dy.
    \label{f_covario_eq}
\end{equation}
\label{f_covario}
\end{definition}
We would like to  emphasize the difference between \eqref{f_covario_eq} and \eqref{e:mu_covario}. Indeed, observe that, when $f=\phi$, we get
$$
g_{\mu,\phi}(K,x)=\int_{K\cap(K+x)} \phi(y-x)\phi(y)dy 
,$$
which is different from $g_{\mu, K}(x)$ and cannot be obtained by considering a measure with the density $\phi^2$. More generally, \eqref{f_covario_eq} is a generalization of \eqref{e:mu_covario}, which cannot be studied by simply considering a density of $f\phi$, as can be seen, in particular, from the next theorem.
\begin{theorem}
\label{f_brightness}
Fix $K\in\conbod$. Let $f$ be a differentiable almost everywhere, non-negative function that is bounded on $K$ and $\mu\in\Lambda$ with density $\phi$ that is Lipschitz in a domain $\Omega$ containing $K$. Then, one has
\begin{equation*}
    \diff{g_{\mu,f}(K,r\theta)}{r}\bigg|_{r=0}=\frac{1}{2}\int_K\langle \left[f(y)\nabla\phi(y)-\phi(y)\nabla f(y) \right],\theta\rangle dy-h_{\Pi_{\mu,K}f}(\theta), \quad \theta\in\s^{n-1}.
\end{equation*}
\end{theorem}
\begin{proof}
We must compute the limit
\begin{equation*}
\begin{split}
    \diff{g_{\mu,f}(K,r\theta)}{r}\bigg|_{r=0}&=\lim_{r\to 0^+}\frac{1}{r}\left(\int_{K\cap(K+r\theta)} f(y-r\theta)\phi(y)dy-\int_K f(y)\phi(y)dy\right).
\\
&=\lim_{r\to 0^+}\frac{1}{r}\left(\int_{K\cap( K-r\theta)}f(y)\phi(y+r\theta)dy-\int_K f(y)\phi(y)dy\right).
\end{split}
\end{equation*}
Next, since $\phi$ is Lipschitz on $\Omega$, and particularly on $K$, we have, via Lemma~\ref{l:rade} that $\phi$ is differentiable almost everywhere in $K$. In particular, $\phi$ is Lipschitz radially and hence its radial derivative is bounded almost everywhere (from the classical fact that one-dimensional Lipschitz functions have a derivative bounded almost everywhere \cite{Stein}). Since $f$ is also bounded, we can use dominated convergence to write
$$\lim_{r\to 0^+} \frac{1}{r}\left(\int_K f(y)\phi (y+r\theta)dy-\int_K f(y)\phi(y)dy\right)=\int_K \langle \nabla \phi(y),\theta\rangle f(y)dy.$$
Hence, we may write
\begin{align*}
    \diff{g_{\mu,f}(K,r\theta)}{r}\bigg|_{r=0}
    &=\int_K \langle \nabla \phi(y),\theta\rangle f(y)dy+\lim_{r\to 0^+}\frac{1}{r}\left(\int_{K\cap(K-r\theta)}f(y)\phi(y+r\theta)dy-\int_K f(y)\phi(y+r\theta)dy\right)
    \\
    &=\int_K \langle \nabla \phi(y),\theta\rangle f(y)dy-\lim_{r\to 0^+}\frac{1}{r}\int_{K\setminus (K-r\theta)}f(y)\phi(y+r\theta)dy
    \\
    &=\int_K \langle \nabla \phi(y),\theta\rangle f(y)dy-\lim_{r\to 0^+}\frac{1}{r}\int_{K\setminus (K-r\theta)}f(y)\phi(y)dy,
\end{align*}
where the last equality follows from $\phi$ being Lipschitz implying
$|\phi(y+r\theta)-\phi(y)|\leq \|\phi\|_{\text{Lip}(\Omega)}r.$
Furthermore, by hypothesis we have that $f$ is bounded on $K$ by some $M>0$. Hence, we see that
\begin{align*}
\bigg|\frac{1}{r}\int_{K\setminus (K-r\theta)}f(y)\phi(y+r\theta)dy-\frac{1}{r}\int_{K\setminus (K-r\theta)}f(y)\phi(y)dy\bigg|&\leq\frac{1}{r}\int_{K\setminus (K-r\theta)}\big|\phi(y+r\theta)-\phi(y)\big|f(y)dy
\\
&\leq M\|\phi\|_{\text{Lip}(\Omega)}\vol_n(K\setminus (K-r\theta))
\end{align*}
which goes to zero as $r\to 0$. Next perform a change of variables to obtain
\begin{align*}
    \diff{g_{\mu,f}(K,r\theta)}{r}\bigg|_{r=0}=\int_K \langle \nabla \phi(y),\theta\rangle f(y)dy-\lim_{r\to 0^+}\frac{1}{r}\int_{(K+r\theta)\setminus K}f(y+r\theta)\phi(y+r\theta)dy.
\end{align*}

From here, since $f$ is bounded and $\phi$ is Lipschitz on $K$ and $K+r\theta$, we can invoke Lemma~\ref{l:squeeze} with $\Psi=\phi\times f$ and $q(r)=r$. Thus, we obtain
\begin{align*}\diff{g_{\mu,f}(K,r\theta)}{r}\bigg|_{r=0}&=\int_K \langle \nabla \phi(y),\theta\rangle f(y)dy-\frac{1}{2}\int_K \langle\nabla(f(y)\phi(y)),\theta\rangle dy-h_{\Pi_{\mu,K} f}(\theta)
\\
&=\frac{1}{2}\int_K\langle \left[f(y)\nabla\phi(y)-\phi(y)\nabla f(y) \right],\theta\rangle dy-h_{\Pi_{\mu,K} f}(\theta).
\end{align*}
\end{proof}
Motivated by Theorem~\ref{f_brightness}, we define the following vector:
\begin{equation*}
    \tau_{\mu,f,K}=\frac{1}{2}\int_K \left(f(y)\nabla\phi(y)-\phi(y)\nabla f(y) \right) dy.
\end{equation*}
We say $f$ is projective with respect to $\mu$ and $K$, or merely $\mu$-projective, if $\tau_{\mu,f,K}=0.$ For ease of notation, we write the convex body $\Pi_{\mu,K} f-\tau_{\mu,f,K}$ whose support function is given by
\begin{equation}\label{e:supp_f_mu}
    h_{\Pi_{\mu,K} f-\tau_{\mu,f,K}}(\theta)=h_{\Pi_{\mu,K} f}(\theta)-\langle\tau_{\mu,f,K},\theta \rangle,
\end{equation}
and so we have that $f$ is $\mu$-projective if, and only if, $\Pi_{\mu,K} f-\tau_{\mu,f,K}=\Pi_{\mu,K} f.$
\begin{remark}
We can write the result of Theorem~\ref{f_brightness} as
$$\diff{g_{\mu,f}(K,r\theta)}{r}\bigg|_{r=0}=-h_{\Pi_{\mu,K} f-\tau_{\mu,f,K}}(\theta).$$
Furthermore, since $g_{\mu,f} (K,r\theta)$ is decreasing as a function of $r$, we have the derivative is non-positive. Thus, $h_{\Pi_{\mu,K} f-\tau_{\mu,f,K}}(\theta) \geq 0$ for every $\theta\in\s^{n-1}$. From convexity, this means that $0\in \Pi_{\mu,K} f-\tau_{\mu,f,K}.$ So, the vector $\tau_{\mu,f,K}$ does not separate $\Pi_{\mu,K} f$ from the origin.
\end{remark}
We see that $g_{\mu,f}(K,\cdot)$ is linear in $f$, and so is the result of Theorem~\ref{f_brightness}. Hence, we obtain the following corollary. We use the standard notation $f^+(x)=\max\{0,f(x)\}$ and $f^{-}(x)=\max\{0,-f(x)\}.$
\begin{corollary}
    Let $f$ be a differentiable almost everywhere function bounded on $K$ and $\mu\in\Lambda$ with density $\phi$ that is Lipschitz in a domain $\Omega$ containing $K$. Then, one has
    $$\diff{g_{\mu,f}(K,r\theta)}{r}\bigg|_{r=0}=\frac{1}{2}\int_K\langle \left[f(y)\nabla\phi(y)-\phi(y)\nabla f(y) \right],\theta\rangle dy-\frac{1}{2}\int_{\partial K}|\langle\theta,u \rangle| f(u)\phi(u) du, \quad \theta\in\s^{n-1}.$$
\end{corollary}
\begin{proof}
Write $f=f^{+}-f^{-}$, use linearity to obtain $g_{\mu,f}(K,\cdot)=g_{\mu,f^{+}}(K,\cdot)-g_{\mu,f^{-}}(K,\cdot),$ which then yields 
$$\diff{g_{\mu,f}(K,r\theta)}{r}\bigg|_{r=0}=\left(-h_{\Pi_{\mu,K} f^{+}-\tau_{\mu,f^{+},K}}(\theta)\right)-\left(-h_{\Pi_{\mu,K} f^{-}-\tau_{\mu,f^{-},K}}(\theta)\right),$$
which equals our claim via our the linearity of the integral in \eqref{e:weak_supp_f_mu}.
\end{proof}

\begin{remark}
If $f$ is also Lipschitz on a domain containing $K$, we can set $f=\phi$, and obtain that $\phi$ is $\mu$-projective, i.e.
$$\diff{g_{\mu,\phi}(K,r\theta)}{r}\bigg|_{r=0}=\diff{}{r}\left(\int_{K\cap(K+x)} \phi(y-x)\phi(y)dy\right)\bigg|_{r=0}=-h_{\Pi_{\mu,K} \phi}(\theta).$$
\end{remark}

\section{The Case of Log-Concave Measures}
\label{sec:log}
Next, we obtain a Zhang-type inequality for projection bodies $\Pi_{\mu,K} f-\tau_{\mu,f,K}$, whose support function is defined in \eqref{e:supp_f_mu}, when $\mu$ has a concavity similar to that of the logarithm function.
\begin{theorem}
Consider $K\in\conbod$ and $\mu\in\Lambda$ such that the density of $\mu$ is locally Lipschitz. Take a bounded function $f:\R^n\to \R^{+}$ that is differentiable almost everywhere and satisfies $\|f\|_{L^1(\mu,K)}\in (0,\infty)$. Now, suppose $Q:(0, \infty)\rightarrow \R$ is an invertible, increasing, differentiable function such that $\lim_{r\to 0^+}Q(r)\in [-\infty,\infty)$ and $Q\circ g_{\mu,f }$ is concave. Then, if $Q^\prime\left(\|f\|_{L^1(\mu,K)}\right)\neq 0$ we have
\begin{equation*}
    \int_{DK}g_{\mu,f}(K,x)dx\leq \frac{n\vol_n\left(\left(\Pi_{\mu,K} f-\tau_{\mu,f,K}\right)^\circ\right)}{\left[Q^\prime(\|f\|_{L^1(\mu,K)})\right]^{n}}
    \int_0^{\infty}Q^{-1}\left(Q\left(\|f\|_{L^1(\mu,K)}\right)-t\right)t^{n-1}dtd\theta.
\end{equation*}
\label{Qf_theorem}
\end{theorem}

\begin{proof}
Begin by applying Lemma~\ref{l:concave} to $Q\circ g_{\mu,f}(K,\cdot)$ to obtain
$$Q(0)<Q\circ g_{\mu,f}(K,r\theta)\leq Q\left(\|f\|_{L^1(\mu,K)}\right)-Q^\prime\left(\|f\|_{L^1(\mu,K)}\right)h_{\Pi_{\mu,K} f-\tau_{\mu,f,K}}(\theta)r.$$
This implies
\begin{equation}\label{e:g_infty_ineq}
    0\leq g_{\mu,f}(K,r\theta)\leq Q^{-1}\left(Q\left(\|f\|_{L^1(\mu,K)}\right)-Q^\prime\left(\|f\|_{L^1(\mu,K)}\right)h_{\Pi_{\mu,K} f-\tau_{\mu,f,K}}(\theta)r\right).
\end{equation}
We now directly compute:
\begin{align*}
    \int_{DK}g_{\mu,f}(K,x)dx&=\int_{\s^{n-1}}\int_0^{\rho_{DK}(\theta)}g_{\mu,f}(K,r\theta)r^{n-1}drd\theta
    \\
    &\leq\int_{\s^{n-1}}\int_0^{\rho_{DK}(\theta)}Q^{-1}\left(Q\left(\|f\|_{L^1(\mu,K)}\right)-Q^\prime\left(\|f\|_{L^1(\mu,K)}\right)h_{\Pi_{\mu,K} f-\tau_{\mu,f,K}}(\theta)r\right)r^{n-1}drd\theta,
\end{align*}
where \eqref{e:g_infty_ineq} was invoked. Setting $t=Q^\prime(\|f\|_{L^1(\mu,K)})h_{\Pi_{\mu,K} f-\tau_{\mu,f,K}}(\theta)r$ yields
\begin{align*}    
\int_{DK}g_{\mu,f}(K,x)dx\leq &\frac{1}{\left[Q^\prime(\|f\|_{L^1(\mu,K)})\right]^{n}} \int_{\s^{n-1}}h^{-n}_{\Pi_{\mu,K} f-\tau_{\mu,f,K}}(\theta)d\theta
\\
    &\times\int_0^{Q^\prime(\|f\|_{L^1(\mu,K)})h_{\Pi_{\mu,K} f-\tau_{\mu,f,K}}(\theta)\rho_{DK}(\theta)}Q^{-1}\left(Q\left(\|f\|_{L^1(\mu,K)}\right)-t\right)t^{n-1}dtd\theta.
\end{align*}
Now, the fact that $Q$ is increasing yields $Q^\prime\left(\|f\|_{L^1(\mu,K)}\right)>0$ and so $$0\leq Q^\prime\left(\|f\|_{L^1(\mu,K)}\right)h_{\Pi_{\mu,K} f-\tau_{\mu,f,K}}(\theta)\rho_{DK}(\theta)<\infty.$$ Hence, 
\begin{align*}
    \int_{DK}g_{\mu,f}(K,x)dx\leq \frac{1}{\left[Q^\prime(\|f\|_{L^1(\mu,K)})\right]^{n}} \int_{\s^{n-1}}h^{-n}_{\Pi_{\mu,K} f-\tau_{\mu,f,K}}(\theta)d\theta
    \int_0^{\infty}Q^{-1}\left(Q\left(\|f\|_{L^1(\mu,K)}\right)-t\right)t^{n-1}dtd\theta.
\end{align*}
Identifying the first integral on the right-hand side as $n\vol_n\left(\left(\Pi_{\mu,K} f-\tau_{\mu,f,K}\right)^\circ\right)$ yields the result. 
\end{proof}

\begin{remark}
In Theorem~\ref{Qf_theorem}, if $Q^\prime\left(\|f\|_{L^1(\mu,K)}\right)= 0$, we obtain:
$$\int_{DK}g_{\mu,f}(K,x)dx\leq \|f\|_{L^1(\mu,K)}\vol_n\left(DK\right).$$
\end{remark}

\begin{corollary}
Suppose $Q:(0, \infty)\rightarrow \R$ is an invertible, increasing, differentiable function such that $\lim_{r\to 0^+}Q(r)\in [-\infty,\infty).$ Let $K\in\conbod$ and $\mu\in\Lambda$ be such that $\mu(K) > 0$ and $\mu$ is $Q$-concave with $\mu$ having a locally Lipschitz density. Then, if $Q^\prime\left(\mu(K)\right)\neq 0$, one has:
\begin{equation}\label{e:g_infty_main_ineq}
    \vol_n(K)\leq \frac{n\vol_n\left(\left(\Pi_{\mu}K-\eta_{\mu,K}\right)^\circ\right)}{\mu(K)(Q^\prime(\mu(K)))^n} \int_0^{\infty}Q^{-1}\left(Q(\mu(K))-t\right)t^{n-1}dt,
\end{equation}
and, if $Q^\prime\left(\mu(K)\right)=0$, one merely obtains $\vol_n(K)<\vol_n(DK).$
\label{Q_theorem}
\end{corollary}

\begin{proof}
From Lemma~\ref{l:covario_concave}, one has that $g_{\mu,K}$ is also $Q$-concave. Thus, set $f=\chi_K$ in Theorem~\ref{Qf_theorem} and use \eqref{int_g_mu}.
\end{proof}
\noindent We can now obtain results for log-concave measures. We first need the following lemma.

\begin{lemma}
\label{l:log}
Let $\mu$ be a locally finite, regular, log-concave Borel measure on $\R^n.$ Then, $\mu$ has a locally Lipschitz density.
\end{lemma}
\begin{proof}
From Borell's classification on concave measures \cite{Bor}, a locally finite and regular Borel measure is log-concave on Borel subsets of $\R^n$ if, and only if,  $\mu$ has a density $\phi(x)$ that is log-concave, i.e. $\phi(x)=Ae^{-\psi(x)},$ where $A>0$ and $\psi:\R^n\to\R^+$ is convex. It is well known that if a function is convex, then it is locally Lipschitz and finite on compact sets (see e.g. \cite{wayne}). Consider a compact set $U \subset \R^n$ and two vectors  $\xi_1, \xi_2\in U.$ One can readily check that $|e^{-a}-e^{-b}|\le e^{-\min\{a,b\}}|a-b|$ for any $a,b \in \R.$ Consequently, we get, with $M=e^{-\min_{x\in U}\psi(x)},$
$$
|\phi(\xi_1)-\phi(\xi_2)|=|Ae^{-\psi(\xi_1)}-Ae^{-\psi(\xi_2)}|\le
MA|\psi(\xi_1)-\psi(\xi_2)|\le MA\|\psi\|_{\text{Lip}}|\xi_1-\xi_2|.
$$
\end{proof}

\begin{theorem} \label{t:log}
Let $K\in\conbod$ and $\mu\in\Lambda$ a log-concave measure such that $\mu(K) >0.$ Then,
\begin{equation*}
    \frac{1}{n!} \leq \frac{\mu^{n}(K) \vol_n\left(\left(\Pi_{\mu}K-\eta_{\mu,K}\right)^\circ\right)}{\vol_n(K)}.
\end{equation*}
\end{theorem}
\begin{proof}
From Lemma~\ref{l:log}, $\mu$ has locally Lipschitz density. Setting $Q(x)=\ln(x)$ in Corollary~\ref{Q_theorem} then yields
\begin{align*}
    \vol_n(K)\leq n \vol_n\left(\left(\Pi_{\mu}K-\eta_{\mu,K}\right)^\circ\right)\mu^n(K)\int_0^{\infty}e^{-t}t^{n-1}dt.
\end{align*}
\end{proof}

\begin{remark}
The Lebesgue measure is log-concave, but, since $\frac{1}{n!}\leq\frac{1}{n^n}{2n \choose n}$, Theorem~\ref{t:log} is weaker than \eqref{e:Zhang_ineq} for the Lebesgue measure.
\end{remark}

For the remainder of this section, we concern ourselves with the Gaussian measure.  We need the following Gaussian analogue of the Brunn-Minkowski inequality, the Ehrhard inequality: 
For $0<t<1,$ Borel sets $K$ and $L$ in $\R^{n}$, and the Gaussian measure $\gamma_n$, we have
\begin{equation}\label{e:Ehrhard_ineq}
\Phi^{-1}\left(\gamma_{n}((1-t) K+tL)\right) \geq(1-t) \Phi^{-1}\left(\gamma_{n}(K)\right)+\lambda \Phi^{-1}\left(\gamma_{n}(L)\right),
\end{equation}
i.e. $\Phi^{-1}\circ\gamma_n$ is concave, where $\Phi(x)=\gamma_{1}((-\infty, x))$. It was first proven by Ehrhard for the case of two closed, convex sets (\cite{EHR1,EHR2}).  Latala \cite{Lat} generalized Ehrhard’s result to the case of arbitrary Borel K and convex L; the general case for two Borel sets of the Ehrhard's inequality was proven by Borell \cite{Bor3}.


Since $\Phi$ is strictly log-concave, the above further implies that the Gaussian measure is $\log$-concave (this also follows directly from the Pr{\'e}kopa-Leindler inequality  \cite{PreL1,PreL2}), i.e.
for $0<t<1$ and closed convex sets $K$ and $L$ in $\mathbb{R}^{n}$,
\begin{equation}\label{e:Log_gamma_concave}
\gamma_{n}((1-t) K+tL) \geq \gamma_{n}(K)^{1-t} \gamma_{n}(L)^{t}.
\end{equation}

\noindent We can immediately apply \eqref{e:Log_gamma_concave} with Theorem~\ref{t:log} to get that for $K\in\conbod_0$ we have
$$\frac{1}{n!} \leq \frac{\gamma_n^{n}(K) \vol_n\left(\left(\Pi_{\gamma_n}K-\eta_{\gamma_n,K}\right)^\circ\right)}{\vol_n(K)}.$$
But we can use \eqref{e:Ehrhard_ineq} to get a sharper bound. In what follows, it will be more convenient to consider the reciprocal of the above, namely
\begin{equation}\label{eq:loggaus}
\frac{\vol_n(K)}{\gamma^n_n(K)\vol_n\left(\left(\Pi_{\gamma_n} K-\eta_{\gamma_n,K}\right)^\circ\right)}\leq n!.
\end{equation}

\begin{corollary}
\label{cor:applying_Ehrhardt}
Consider $K\in\conbod$ and set $ x=\Phi^{-1}(\gamma_n(K)).$ Then, one has
\begin{equation}\label{e:applying_Ehrhardt}
    \frac{\vol_n(K)}{\gamma^n_n(K)\vol_n\left(\left(\Pi_{\gamma_n} K-\eta_{\gamma_n,K}\right)^\circ\right)}\leq \frac{e^{-nx^2/2}}{(2\pi\Phi(x)^2)^{(n+1)/2}}\int_{0}^{\infty} z^n e^{-(z-x)^2/2}dz.
\end{equation}
\end{corollary}
\begin{proof}
From \eqref{e:Ehrhard_ineq}, we have that $\Phi^{-1}\circ\gamma_n$ is concave on Borel sets, and so we have that $\Phi^{-1}\circ g_{{\gamma_n},K}$ is concave on $\conbod$. Hence, we can apply Corollary~\ref{Q_theorem} with $\mu=\gamma_n$ and $Q=\Phi^{-1}.$ From rudimentary inverse function theory, we have
$$
\diff{}{x}\Phi^{-1}(x)\bigg|_{x=\gamma_n(K)}=\left[\diff{}{y}\Phi(y)\bigg|_{y=\Phi^{-1}(\gamma_n(K))}\right]^{-1}.
$$ By setting $\Phi(y)=(2\pi)^{-1/2}\int_{-\infty}^{y}e^{-r^2/2}dr$, we obtain
\begin{align*}
    \vol_n(K)\leq\frac{n\vol_n\left(\left(\Pi_{\gamma_n} K-\eta_{\gamma_n,K}\right)^\circ\right)}{(2\pi)^{n/2}\gamma_n(K)}e^{-n(\Phi^{-1}(\gamma_n(K))^2/2}\int_{0}^{\infty} \Phi[\Phi^{-1}(\gamma_n(K))-t] t^{n-1} d t.
\end{align*}
For simplicity, we set $x=\Phi^{-1}\left(\gamma_n(K)\right).$ Next, explicitly inserting $\Phi(x-t)$ yields
\begin{align*}
    \frac{\vol_n(K)}{\gamma^{n}_n(K)\vol_n\left(\left(\Pi_{\gamma_n} K-\eta_{\gamma_n,K}\right)^\circ\right)}\leq\frac{n}{\gamma^{n+1}_n(K)(2\pi)^{(n+1)/2}}e^{-nx^2/2}\int_{0}^{\infty} \int_{-\infty}^{x-t}e^{-y^2/2} t^{n-1}  dy d t.
\end{align*}
Using Fubini's yields
\begin{align*}
    \frac{\vol_n(K)}{\gamma^{n}_n(K)\vol_n\left(\left(\Pi_{\gamma_n} K-\eta_{\gamma_n,K}\right)^\circ\right)}&\leq\frac{n}{\gamma^{n+1}_n(K)(2\pi)^{(n+1)/2}}e^{-nx^2/2}\int_{-\infty}^{x}\int_{0}^{x-y} e^{-y^2/2} t^{n-1}  dt dy
    \\
    &=\frac{e^{-nx^2/2}}{\gamma^{n+1}_n(K)(2\pi)^{(n+1)/2}}\int_{-\infty}^{x} e^{-y^2/2} (x-y)^n  dy.
\end{align*}
Letting $z=x-y$ yields the result.
\end{proof}

We would like now to verify that the right-hand side of \eqref{e:applying_Ehrhardt} is sharper than \eqref{eq:loggaus} by showing that, in general, a stronger concavity assumption implies a stronger Zhang's type inequality in Theorem~\ref{Qf_theorem}. Indeed, one has that
$$\frac{\vol_n(K)}{\gamma^n_n(K)\vol_n\left(\left(\Pi_{\gamma_n} K-\eta_{\gamma_n,K}\right)^\circ\right)}\leq \frac{e^{-nx^2/2}}{(2\pi\Phi(x)^2)^{(n+1)/2}}\int_{0}^{\infty} z^n e^{-(z-x)^2/2}dz\leq n!; \quad x=\Phi^{-1}(\gamma_n(K)),$$
where the first inequality is a restatement of Corollary~\ref{cor:applying_Ehrhardt} and the second inequality follows from the proposition below, by setting $a=\gamma_n(K),$ $F=\Phi^{-1},$ $Q=\ln$, and multiplying through by $\frac{n}{\gamma^{n+1}_n(K)}$. Taking reciprocal then yields
$$\frac{1}{n!}\leq \frac{e^{nx^2/2}(2\pi\Phi(x)^2)^{(n+1)/2}}{\int_{0}^{\infty} z^n e^{-(z-x)^2/2}dz} \leq  \frac{\gamma^n_n(K)\vol_n\left(\left(\Pi_{\gamma_n} K-\eta_{\gamma_n,K}\right)^\circ\right)}{\vol_n(K)}; \quad x=\Phi^{-1}(\gamma_n(K)). $$
\begin{proposition}
Consider two invertible, increasing, differentiable functions $Q, F: (0,\infty)\rightarrow\R$  such that $Q(F^{-1})$ is concave. Then, for all $a>0$:
$$
  \frac{1}{(F^\prime(a))^n} \int_0^{\infty}F^{-1}\left(F(a)-t\right)t^{n-1}dt \le    \frac{1}{(Q^\prime(a))^n} \int_0^{\infty}Q^{-1}\left(Q(a)-t\right)t^{n-1}dt.
$$
\label{pr:comparison}
\end{proposition}
\begin{proof} The inequality in Proposition~\ref{pr:comparison} is
equivalent to
$$
 \int_0^{\infty}F^{-1}\left(F(a)-F'(a)t\right)t^{n-1}dt \le     \int_0^{\infty}Q^{-1}\left(Q(a)-Q'(a)t\right)t^{n-1}dt.
$$
Thus, it is sufficient to prove that 
\begin{equation}\label{eq:cocavitycomp}
 F^{-1}\left(F(a)-F'(a)t\right)\le Q^{-1}\left(Q(a)-Q'(a)t\right), \,\,\,\, \forall \, t \ge 0,
\end{equation}
or 
$$
Q (F^{-1})\left(F(a)-F'(a)t\right)\le Q(a)-Q'(a)t , \,\,\,\, \forall \, t \ge 0.
$$
Let $G(t)=Q (F^{-1})\left(F(a)-F'(a)t\right)$.
We notice that $G(t)$ is a concave function with $G(0)=Q(a)$ and $G'(0)=-Q'(a)$, which  proves \eqref{eq:cocavitycomp}.
\end{proof}

\section{Zhang's Inequality for Two Measures}
\label{super_zhang_sec}
We now work towards a version of Zhang's inequality for  $\Pi_{\mu,K} f-\tau_{\mu,f,K}.$  We start by defining an analogue of \eqref{e:mu_average} for a function and two different measures:
\begin{definition}
Suppose $\mu$ and $\nu$ are Borel measures on $\R^n$. The $\nu$-translation of an integrable function $f\in L^1(\mu,K)$, K a compact, non-empty set, averaged with respect to $\mu$, is given by
$$\nu_{\mu}(f,K)=\frac{1}{\|f\|_{L^1(\mu,K)}}\int_{DK}g_{\mu,f}(K,x)d\nu(x).$$
In the case when $f=\chi_K$, $K\in\conbod$, we say the $\nu$-translation of a convex body $K$ averaged with respect to $\mu$ is given by
$$
\nu_{\mu}(K)=\frac{1}{\mu(K)}\int_{DK}g_{\mu,K}(x)d\nu(x)=\frac{1}{\mu(K)}\int_{K}\nu(y-K)d\mu(y),
$$
where the far right equality follows analogously to Example~\ref{int_g_mu}.
\label{nu_mu_avg}
\end{definition} 
We see that Definition~\ref{nu_mu_avg} serves as a generalization Definition~\ref{int_g_mu} and \eqref{e:mu_average}. Notice Definition~\ref{nu_mu_avg} and \eqref{e:mu_set_incl} yields, for $\nu$ a Borel measure and $\mu\in\Lambda$ that is $F$-concave, where $F$ is an increasing, non-negative, differentiable function, that
	\begin{equation*}\label{eq:av2}
	    \nu_\mu(K)\leq \nu(DK) \leq \nu\left(\frac{F(\mu(K))}{F^\prime(\mu(K))}\left(\Pi_{\mu}K-\eta_{\mu,K}\right)^\circ\right).
	\end{equation*}
To obtain sharper constants, we shall again suppose that $\nu\in\Lambda_{\text{rad}}.$ We now state the main result of this section.
\begin{theorem}[Functional Version of Zhang's Inequality for Two Measures]
Consider $K\in\conbod$, $\mu\in\Lambda$, $\nu\in\Lambda_{\text{rad}}$, such that the density of $\mu$ is locally Lipschitz. Let $f:\R^n \to \R^{+}$ be a differentiable almost everywhere and bounded function in $L^1(\mu,K)\cap L^1(\mu,\partial K)$. Suppose $F$ is an increasing, non-negative, and invertible function such that $F\circ g_{\mu,f}$ is concave. Then,
$$
    \|f\|_{L^1(\mu,K)}\nu_{\mu}(f)\leq n \nu\left(\frac{F(\|f\|_{L^1(\mu,K)})}{F^\prime(\|f\|_{L^1(\mu,K)})}\left(\Pi_{\mu,K} f-\tau_{\mu,f,K}\right)^\circ\right)
    \left(\int_{0}^1F^{-1}\left(F(\|f\|_{L^1(\mu,K)})t\right)(1-t)^{n-1}dt\right).
$$
Equality occurs if, and only if, the following are true:
\begin{enumerate}
    \item If $\varphi$ is the density of $\nu,$ then, for each $\theta\in\s^{n-1}$, $\varphi(r\theta)$ is independent of $r,$
    \item for each $\theta\in\s^{n-1},$ $F\circ g_{\mu,f}(K,r\theta)$ is an affine function in the variable $r$ for $r\in[0,\rho_{DK}(\theta)]$, and
    \item  $f$ is projective with respect to $\mu$ and $K$ such that $$DK=\frac{F(\|f\|_{L^1(\mu,K)})}{F^\prime(\|f\|_{L^1(\mu,K)})}\Pi^\circ_{\mu,K}f.$$
\end{enumerate}

\label{Ff_theorem}
\end{theorem}
\begin{proof}
We shall invoke Lemma~\ref{t:chak}, with $L=DK$, $q=F^{-1}$ and the function $F\circ g_{\mu,f}(K,\cdot)$. So, we directly calculate: $F(g_{\mu,f}(K,0))=F(\|f\|_{L^1(\mu,K)})$ and $$\diff{F(g_{\mu,f}(K,r\theta))}{r}\bigg|_{r=0}=\diff{g_{\mu,f}(K,r\theta)}{r}\bigg|_{r=0}\diff{F(y)}{y}\bigg|_{y=g_{\mu,f}(K,0)}=-h_{\Pi_{\mu,K} f-\tau_{\mu,f,K}}(\theta)F^\prime(\|f\|_{L^1(\mu,K)}).$$
Then, we set $$z(\theta)=\frac{F(\|f\|_{L^1(\mu,K)})}{F^\prime(\|f\|_{L^1(\mu,K)})}\rho_{\left(\Pi_{\mu,K} f-\tau_{\mu,f,K}\right)^\circ}(\theta)=\rho_{\frac{F(\|f\|_{L^1(\mu,K)})}{F^\prime(\|f\|_{L^1(\mu,K)})}\left(\Pi_{\mu,K} f-\tau_{\mu,f,K}\right)^\circ}(\theta).$$
For the constant $\beta$, we have that
$$\beta=n\int_{0}^1(F^{-1}) \left(F(\|f\|_{L^1(\mu,K)}) t\right)(1-t)^{n-1}dt.$$ 
We now compute:

\begin{flalign*}
    &\|f\|_{L^1(\mu,K)}\nu_{\mu}(f)=\int_{DK}g_{\mu,f}(K,x)d\nu(x)=\int_{DK} F^{-1}\left[F\circ g_{\mu,f}(K,x)\right]d\nu(x)
    \\
    &\leq n\nu\left(\frac{F(\|f\|_{L^1(\mu,K)})}{F^\prime(\|f\|_{L^1(\mu,K)})}\left(\Pi_{\mu,K} f-\tau_{\mu,f,K}\right)^\circ\right)\int_{0}^1F^{-1}\left(F(\|f\|_{L^1(\mu,K)})t\right)(1-t)^{n-1}dt,
\end{flalign*}
where the inequality follows from equating $F^{-1}\left[F\circ g_{\mu,f}(K,\cdot)\right]$ with $q\circ f$ in Lemma~\ref{t:chak}. Thus, we have established our claim. The equality conditions are inherited from Lemma~\ref{t:chak}. The first two are immediate. As for the third, $\rho_{DK}(\theta)=z(\theta)$ implies $DK=\frac{F(\|f\|_{L^1(\mu,K)})}{F^\prime(\|f\|_{L^1(\mu,K)})}\left(\Pi_{\mu,K} f-\tau_{\mu,f,K}\right)^\circ.$ But, since $DK$ and $\Pi_{\mu,K} f$ are symmetric, this can only happen if $\tau_{\mu,f,K}=0.$
\end{proof}

We would like to remark that, in a similar manner to Proposition~\ref{pr:comparison},
 a stronger concavity assumption in Theorem~\ref{Ff_theorem} gives a stronger Zhang's Inequality for Two Measures:
\begin{proposition}
Consider two invertible, increasing, non-negative, differentiable functions $Q, F: (0,\infty)\rightarrow\R^+$ such that  $Q(F^{-1})$ is concave. Then, for a star-shaped body $L$ and for all $\nu\in\Lambda_{\text{rad}}$ and  $a>0$, one has
\begin{align*}
   \nu\left(\frac{F(a)}{F^\prime(a)} L\right)
    \int_{0}^1F^{-1}\left(F(a)t\right)(1-t)^{n-1}dt \le 
   \nu\left(\frac{Q(a)}{Q^\prime(a)} L\right)
    \int_{0}^1Q^{-1}\left(Q(a)t\right)(1-t)^{n-1}dt.
\end{align*}
\end{proposition}

\begin{proof} Let $\varphi$ be the density of $\nu$, then we integrate via polar coordinates to obtain
$$
\nu\left(\frac{F(a)}{F^\prime(a)} L\right)
    \int_{0}^1F^{-1}\left(F(a)t\right)(1-t)^{n-1}dt = \int_{\s^{n-1}}
    \int_0^{\frac{F(a)}{F^\prime(a)}\rho_L(\theta)} \varphi(r) r^{n-1}dr d\theta
    \int_{0}^1F^{-1}\left(F(a)(1-t)\right)t^{n-1}dt 
$$
$$
= \int_{\s^{n-1}}
    \int_0^{\rho_L(\theta)} \varphi\left(\frac{F(a)}{F^\prime(a)}r\right) r^{n-1}dr d\theta
    \int_{0}^1F^{-1}\left(F(a)- F'(a)t)\right)t^{n-1}dt $$
    $$
    =  \int_{\s^{n-1}}
    \int_0^{\rho_L(\theta)} \int_{0}^1 \varphi\left(\frac{F(a)}{F^\prime(a)}r\right) 
    F^{-1}\left(F(a)- F'(a)t)\right)t^{n-1} r^{n-1}dt dr  d\theta.
$$
Thus, to complete the proof, it is enough to show that
$$
\varphi\left(\frac{F(a)}{F^\prime(a)}r\right) 
    F^{-1}\left(F(a)t- F'(a)t)\right) \le \varphi\left(\frac{Q(a)}{Q^\prime(a)}r\right) 
    Q^{-1}\left(Q(a)- Q'(a)t)\right).
$$
Using \eqref{eq:cocavitycomp}, we see that it is sufficient to show that
$
\varphi\left(\frac{F(a)}{F^\prime(a)}r\right) 
    \le \varphi\left(\frac{Q(a)}{Q^\prime(a)}r\right). 
$
Using that $\varphi$ is non-decreasing, it is enough to prove that 
$$
\frac{F(a)}{F^\prime(a)} 
    \le \frac{Q(a)}{Q^\prime(a)}. 
$$
Let $a=F^{-1}(b)$. Then, our goal is to prove
\begin{equation}\label{eq:ratio}
\frac{b}{F^\prime(F^{-1}(b))} 
    \le \frac{Q(F^{-1}(b))}{Q^\prime(F^{-1}(b))}. 
\end{equation}
Denote by $G(b)=Q(F^{-1}(b))$ so that $G'(b)=\frac{Q'(F^{-1}(b))}{F'(F^{-1}(b))}$ and (\ref{eq:ratio}) becomes
$G'(b)b \le G(b)$; but, this follows immediately from concavity of $G$ and the fact that $G(0)\geq 0$.  Indeed,
$$
G(b)=G(0)+G'(c)(b-0),
$$
where $c\in [0,b]$. Using the concavity of $G$, we get that $G'(c)\ge G'(b)$, and thus we have:
$$
G(b)=G(0)+G'(c)(b-0) \ge G(0)+G'(b)(b-0) \ge G'(b)b,
$$
which implies the claim.
\end{proof}

\begin{corollary}
[Zhang's Inequality for Two Measures]
Suppose $F$ is an increasing, differentiable, non-negative and invertible function. Consider $K\in\conbod$, $F$-concave $\mu\in\Lambda$, $\nu\in\Lambda_{\text{rad}}$, and further suppose the density of $\mu$ is locally Lipschitz. Then,
\begin{align*}
    \nu_{\mu}(K)\leq\frac{n}{\mu(K)}\nu\left(\frac{F(\mu(K))}{F^\prime(\mu(K))}\left(\Pi_{\mu}K-\eta_{\mu,K}\right)^\circ\right)\int_{0}^1F^{-1}\left(F(\mu(K))t\right)(1-t)^{n-1}dt.
\end{align*}
\label{F_theorem}
Equality occurs if, and only if, the following are true:
\begin{enumerate}
    \item If $\varphi$ is the density of $\nu,$ then, for each $\theta\in\s^{n-1}$, $\varphi(r\theta)$ is independent of $r,$
    \item for each $\theta\in\s^{n-1},$ $F\circ g_{\mu,K}(r\theta)$ is an affine function in the variable $r$ for $r\in[0,\rho_{DK}(\theta)],$ and
    \item  $K$ is $\mu$-projective such that $$DK=\frac{F(\mu(K))}{F^\prime(\mu(K))}\Pi^\circ_{\mu}K.$$
\end{enumerate}

\end{corollary}
\begin{proof}
From Lemma~\ref{l:covario_concave}, $g_{\mu,K}$ is $F$-concave. Set $f=\chi_K$ in Theorem~\ref{Ff_theorem}, for $K\in\conbod.$ The equality conditions are inherited from Lemma~\ref{t:chak}.
\end{proof}

We now obtain a result for $s$-concave measures, $s>0.$ We first need the following lemma, which establishes that $s$-concave measures have locally Lipschitz densities.
\begin{lemma}
\label{l:scon}
Let $\mu$ be a locally finite, regular, $s$-concave Borel measure on $\R^n,$ $s>0.$ Then, $\mu$ has a locally Lipschitz density.
\end{lemma}
\begin{proof}
From Borell's classification on concave measures \cite{Bor}, a locally finite and regular Borel measure is $s$-concave on Borel subsets of $\R^n$, $s\in (0,1/n),$ if, and only if, $\mu$ has a density $\phi(x)$ that is $p$-concave, where $p=s/(1-ns)>0.$ Consequently, one has, for $x,y$ in the support of $\phi$ and $t\in[0,1]$ that
$$\phi(tx+(1-t)y)\geq \left(t\phi(x)^p+(1-t)\phi(y)^p\right)^{1/p}\geq \phi(x)^t\phi(y)^{1-t},$$
where the first inequality is from the $p$-concavity of $\phi$ and the second inequality is from Jensen's inequality. Consequently, $\phi$ is also log-concave, and the result follows from Lemma~\ref{l:log}. When $s=1/n,$ then $\phi$ is a constant on its support, and, when $s>1/n,$ $\phi$ is identically zero. In either case, the result follows.
\end{proof}
\noindent We also also need the following characterization of a simplex.
\begin{lemma}
\label{l:simp}
Let $\mu\in\Lambda$ be an $s$-concave measure on Borel measurable sets, $s>0$, such that the the density of $\mu$ is not zero almost everywhere. Then for $K\in\conbod,$ the following are equivalent:
\begin{enumerate}
    \item K is a simplex.
    \item For every $\theta\in\s^{n-1},$ $g^s_{\mu,K}(r\theta)$ is an affine function in the variable $r$ for $r\in[0,\rho_{DK}(\theta)].$
\end{enumerate}
\end{lemma}
\begin{proof}
For $\theta\in\s^{n-1},$ the proof of Lemma~\ref{l:covario_concave} reveals that $g^s_{\mu,K}(r\theta)$ is an affine function in the variable $r$ if, and only if, for every $r_1,r_2\geq 0,$ one has
$$\mu^s((1-\lambda)K\cap (K+r_1\theta)+\lambda K\cap (K+r_2\theta))=(1-\lambda)\mu^s(K\cap (K+r_1\theta)) + \lambda \mu^s(K\cap(K+r_2\theta)).$$
As established by Milman and Rotem \cite[Corollary 2.16]{MR14}, this is true if, and only if, $K\cap(K+r_1\theta)$ is homothetic to $K\cap(K+r_2\theta),$ for every $r_1,r_2 \geq 0$ such that those intersections are non-empty. In particular, Proposition~\ref{p:simp} shows this is true if, and only if, $K$ is a simplex.
\end{proof}

\begin{theorem}[Zhang's Inequality for $s$-concave measures]
Let $\nu$ be a radially non-decreasing measure. If $\mu\in\Lambda$ is a $s$-concave measure on Borel measurable sets, $s>0,$ then, for $K\in\conbod,$ we have that
\begin{equation}
    {{n+s^{-1}}\choose n}\nu_\mu(K)\leq\nu\left(s^{-1}\mu(K)\left(\Pi_{\mu}K-\eta_{\mu,K}\right)^\circ\right). 
\end{equation}
Equality occurs if, and only if, the following are true:
\begin{enumerate}
    \item If $\varphi$ is the density of $\nu,$ then, for each $\theta\in\s^{n-1}$, $\varphi(r\theta)$ is independent of $r,$
    \item K is a $\mu$-projective simplex, and
    \item  $DK=s^{-1}\mu(K)\Pi^\circ_{\mu}K.$
\end{enumerate}

\label{F_s}
\end{theorem}
\begin{proof}
From Lemma~\ref{l:scon}, $\mu$ has Lipschitz density. Next, setting $F(x)=x^{s}$ in Corollary~\ref{F_theorem} yields
$$\nu_\mu(K)\leq n \nu\left(s^{-1}\mu(K)\left(\Pi_{\mu}K-\eta_{\mu,K}\right)^\circ\right)\int_{0}^1 t^{1/s} t^{n-1}dt.$$
If we denote
$$
\binom{n+s^{-1}}{n}=\frac{\Gamma (n+s^{-1}+1)}{n!\Gamma(s^{-1}+1)},
$$
then recognizing the integral is $B(s^{-1}+1,n)=\left[n\binom{n+s^{-1}}{n}\right]^{-1}$
yields the result. The equality conditions are inherited from Lemmas~\ref{t:chak} and \ref{l:simp}.
\end{proof}

\noindent We obtain the following immediate corollary.
\begin{corollary}
\label{cor:s_zhang}
If $\nu=\lambda$, the Lebesgue measure, in Corollary~\ref{F_theorem}, then we have
\begin{equation*}
    \vol_n(K)\leq\frac{n\vol_n\left(\left(\Pi_{\mu}K-\eta_{\mu,K}\right)^\circ\right)}{\mu(K)}\left(\frac{F(\mu(K))}{F^\prime(\mu(K))}\right)^n\int_{0}^{1} F^{-1}[F(\mu(K))t] (1-t)^{n-1} d t.
\end{equation*}
Furthermore, if we have that $F(x)=x^{s}$, $s\in (0,1/n],$ then we have from Theorem~\ref{F_s}
\begin{equation*}
    s^n\binom{n+s^{-1}}{n}\leq\frac{\mu^{n}(K)\vol_n\left(\left(\Pi_{\mu}K-\eta_{\mu,K}\right)^\circ\right)}{\vol_n(K)}.
    \end{equation*}
    In particular, when we also have $\mu=\lambda$, one is allowed to set $s=1/n$ and obtain \eqref{e:Zhang_ineq}.
\end{corollary}
We conclude by stating the following theorem, whose proof is the same as Corollary~\ref{F_theorem} with $r_{\mu,K}$ in place of $g_{\mu,K}$, and is omitted. Recall that the support of $r_{\mu,K}$ is still $DK.$

\begin{theorem}[Polarized Zhang's Inequality for Even Measures on Symmetric Convex Bodies]
Consider a symmetric body  $K\in\conbod_0$, an even measure $\mu\in\Lambda$ with locally Lipschitz density, and a measure $\nu\in\Lambda_{\text{rad}}$. Suppose $R$ is an increasing, non-negative, and invertible function such that $R\circ r_{\mu,K}$ is concave. Then,
\begin{align*}
    \nu_{\mu}(K)\leq\frac{n}{\mu(K)}\nu\left(\frac{R(\mu(K))}{R^\prime(\mu(K))}\Pi^\circ_\mu K\right)\int_{0}^1R^{-1}\left(R(\mu(K))t\right)(1-t)^{n-1}dt.
\end{align*}
Equality occurs if, and only if, the following are true:
\begin{enumerate}
    \item If $\varphi$ is the density of $\nu,$ then, for each $\theta\in\s^{n-1}$, $\varphi(r\theta)$ is independent of $r,$
    \item for each $\theta\in\s^{n-1},$ $R\circ r_{\mu,K}(r\theta)$ is an affine function in the variable $r$ for $r\in[0,\rho_{DK}(\theta)],$ and
    \item   $DK=\frac{R(\mu(K))}{R^\prime(\mu(K))}\Pi^\circ_{\mu}K.$
\end{enumerate}

\label{R_theorem}
\end{theorem}

\noindent Repeating the calculations for Theorem~\ref{F_s}, we obtain the following corollary.
\begin{corollary}[Polarized Zhang's Inequality for $s$-concave, even measures]
Suppose $s>0$ and let $\nu$ be a radially non-decreasing measure. Also, consider an even  measure  $\mu\in\Lambda$ with locally Lipschitz density that is $s$-concave on the class of symmetric convex bodies. Then, for symmetric $K\in\conbod_0,$ one has
\begin{equation*}
    {{n+s^{-1}}\choose n}\nu_\mu(K)\leq\nu\left(s^{-1}\mu(K)\Pi^\circ_\mu K\right).
\end{equation*}
Equality occurs if, and only if, the following are true:
\begin{enumerate}
    \item If $\varphi$ is the density of $\nu,$ then, for each $\theta\in\s^{n-1}$, $\varphi(r\theta)$ is independent of $r,$
    \item for each $\theta\in\s^{n-1},$ $r^s_{\mu,K}(r\theta)$ is an affine function in the variable $r$ for $r\in[0,\rho_{DK}(\theta)],$ and
    \item  $DK=s^{-1}\mu(K)\Pi^\circ_{\mu}K.$
\end{enumerate}
\label{R_s}
\end{corollary}

We obtain the following immediate corollary for the case when $\nu$ is the Lebesgue measure.
\begin{corollary}
If $\nu=\lambda$, the Lebesgue measure, in Theorem~\ref{R_theorem}, then we have
\begin{equation*}
    \vol_n(K)\leq\frac{n\vol_n\left(\Pi^\circ_\mu K\right)}{\mu(K)}\left(\frac{R(\mu(K))}{R^\prime(\mu(K))}\right)^n\int_{0}^{1} R^{-1}[R(\mu(K))t] (1-t)^{n-1} d t.
\end{equation*}
Furthermore, if we have that $R(x)=x^{s}$, $s>0,$ then we have from Corollary~\ref{R_s}
\begin{equation*}
    s^n{{n+s^{-1}}\choose n}\vol_n(K)\leq\mu^{n}(K)\vol_n\left(\Pi^\circ_\mu K\right).
    \end{equation*}
\end{corollary}
\noindent Finally, we obtain the following corollary for when $\mu=\gamma_n$.
\begin{corollary}[Polarized Zhang's Inequality for the Gaussian Measure]
    Consider the case in Corollary~\ref{R_s} for $\mu=\gamma_n$. Then, from \eqref{e:gamma_gaussian}, one is allowed to set $s=1/n$ and obtain, for $K\in\conbod$ such that $K=-K$:
    \begin{equation*}
        {{2n}\choose n}\nu_{\gamma_n}(K)\leq\nu\left(n\gamma_n(K)\Pi^\circ_{\gamma_n} K\right).
    \end{equation*}
    If $\nu=\lambda$, the Lebesgue measure, then we have 
    \begin{equation*}
    \frac{1}{n^n}{{2n}\choose n}\vol_n(K)\leq\gamma_n^{n}(K)\vol_n\left(\Pi^\circ_{\gamma_n} K\right).
    \end{equation*}
\end{corollary}

\section{An Application to a Reverse Isoperimetric Inequality}
\label{sec:iso}
In this section, we shall use the Zhang inequalities of Sections~\ref{sec:log} and \ref{super_zhang_sec} to obtain a reverse isoperimetric inequality. We shall follow a schema by Giannopoulos and Papadimitrakis \cite{GP99}, who had previously done the volume case of what follows. In fact, this particular result in the volume case had been obtained back in 1961 by Petty \cite{petty2} using different methods. Throughout, $\mathrm{Id}=\mathrm{Id}_n$ shall denote the identity operator on $\R^n.$ We will use the concept of isotropic measures applied to $S_{\mu,K}$. To give an example of isotropicity, we show that, if $S_{\mu,K}$ is isotropic, we obtain a much a sharper estimate for the size of $\Pi_\mu K.$
\begin{proposition}
Suppose $S_{\mu,K}$ is isotropic. Then,
$$\frac{\mu(\partial K)}{2n}B_2^n \subset \Pi_\mu K \subset \frac{\mu(\partial K)}{2\sqrt{n}}B_2^n.$$
\end{proposition}
\begin{proof}
It suffices to show that, for every $\theta\in\s^{n-1},$
$$\frac{\mu(\partial K)}{2n} \leq h_{\Pi_\mu K}(\theta) \leq \frac{\mu(\partial K)}{2\sqrt{n}}.$$
But this is immediate from the fact that, for $\theta,u\in\s^{n-1},\, |\langle u,\theta\rangle|^2\leq |\langle u,\theta\rangle|$ and also the Cauchy-Schwarz inequality:
$$\frac{1}{2}\int_{\s^{n-1}}|\langle u,\theta\rangle|^2dS_{\mu,K}(u)\leq h_{\Pi_{\mu}K}(\theta)=\frac{1}{2}\int_{\s^{n-1}}|\langle u,\theta\rangle|dS_{\mu,K}(u)\leq \frac{\sqrt{\mu(\partial K)}}{2}\left(\int_{\s^{n-1}}|\langle u,\theta\rangle|^2dS_{\mu,K}(u)\right)^{\frac{1}{2}};$$
using that
$$\frac{\mu(\partial K)}{n}=\int_{\s^{n-1}}|\langle u,\theta\rangle|^2dS_{\mu,K}(u)$$
by hypothesis completes the proof.
\end{proof}

Next, fix $\mu\in\Lambda$ and $K\in\conbod,$ and, for $A \in SL_n(\R),$ consider
\begin{equation}
\label{e:best_surface}
    I_{\mu,K}(A):=\int_{\s^{n-1}} |Au|dS_{\mu,K}(u).
\end{equation}
We now show that $S_{\mu,K}$ being isotropic occurs precisely when \eqref{e:best_surface} is minimized at $\mathrm{Id}$. Observe that in the volume case, one has $I_{\lambda,K}(A)=\vol_{n-1}(\partial (A^{-t} K))$ \cite[Chapter 8,pg. 308]{gardner_book}. Thus, finding the minimum of $I_{\lambda,K}$ is equivalent to what is called the minimal surface area position of $K$. It is well known \cite{AGA} that an isotropic measure permits the following decomposition of identity.

\begin{lemma}
\label{l:optim}
Let $K\in\conbod$ and consider $\mu\in\Lambda$ such that $\mu(\partial K)\in (0,\infty).$ Then, $S_{\mu,K}$ is an isotropic measure if, and only if, \eqref{e:best_surface} is minimized at $\mathrm{Id}.$ In particular, one obtains the following decomposition of identity: 
\[
\mathrm{Id} = \frac{n}{\mu(\partial K)} \int_{\s^{n-1}}  u \otimes u dS_{\mu,K}(u).
\]
\begin{proof}
Our claim has been previously shown in the case where $\mu=\lambda$ in \cite{GP99}; the general case follows. Indeed, since the function $f_A(u)=|Au|$ is even for every $A\in GL(n)$, one obtains that
$$I_{\mu,K}(A)=\int_{\s^{n-1}}f_A(u)dS_{\mu,K}(u)=\int_{\s^{n-1}}f_{A}(u)dE(u),$$
where $dE(u)=\frac{1}{2}(dS_{\mu,K}(u)+dS_{\mu,K}(-u)).$ Then, one obtains from Minkowski's Existence Theorem \cite[Theorem 8.2.2]{Sh1} that there exists some $L\in\conbod$ such that $dE(u)=dS_{L}(u).$ Clearly, $S_{\mu,K}$ is isotropic if, and only if, $S_{L}$ is isotropic, and $I_{\mu,K}(A)=I_{\lambda,L}(A).$ Therefore, from the proof of Theorem 1 in \cite{petty2}, a minimum exists. Furthermore, $S_L(\s^{n-1})=E(\s^{n-1})=S_{\mu,K}(\s^{n-1})=\mu(\partial K),$ and the claim follows.
\end{proof}

\end{lemma}

Our goal is to obtain a reverse isoperimetric inequality for $\mu(\partial K)$ when $S_{\mu,K}$ is isotropic. We first recall the following result from K. Ball \cite{Ball2}.
\begin{proposition}
\label{p:ball}
 Let $\left\{u_{j}\right\}_{j \leq m}$ be unit vectors in $\mathbb{R}^{n}$ and $\left\{c_{j}\right\}_{j \leq m}$ be positive numbers satisfying
$$
\mathrm{Id}=\sum_{j=1}^{m} c_{j} u_{j} \otimes u_{j}.
$$
Define a norm in $\mathbb{R}^{n}$ by $\|x\|=\sum_{j=1}^{m} \alpha_{j}\left|\left\langle x, u_{j}\right\rangle\right|$, where $\alpha_{j}>0 .$ Let $L$ be the unit ball of $\left(\mathbb{R}^{n},\|\cdot\|\right).$ Then,
$$
\vol_n(L) \leq \frac{2^{n}}{n !} \prod_{j=1}^{m}\left(\frac{c_{j}}{\alpha_{j}}\right)^{c_{j}}.$$
\end{proposition}

\noindent Next, let $K$ be a polytope with facets $F_j, \; j=1,\dots,m$. For $\mu\in\Lambda$ with density $\phi,$ define
$$\mu(F_j)=\int_{F_j}\phi(x)dx.$$ 
Suppose that \eqref{e:best_surface} is minimized at $\mathrm{Id}.$ Consequently, Lemma~\ref{l:optim} gives that
$$\mathrm{Id}=\sum_{j=1}^m\frac{n\mu(F_j)}{\mu(\partial K)}u_j\otimes u_j,$$
where $u_j$ is the outer-unit normal of $F_j$. Additionally, the support function of $\Pi_\mu K$ can be written as
$$h_{\Pi_\mu K}(\theta)=\sum_{j=1}^m\frac{\mu(F_j)}{2}|\langle \theta,u_j \rangle|.$$
From Proposition~\ref{p:ball}, one then obtains that
$$\vol_n(\Pi^\circ_\mu K)\leq \frac{2^n}{n!}\left(\frac{2n}{\mu(\partial K)}\right)^n.$$
By arguing via approximation, and combining with Lemma~\ref{cor:sur}, we get the following.

\begin{theorem}
\label{t:almost_iso}
Let $K$ be a convex body and $\mu\in\Lambda.$ Suppose that \eqref{e:best_surface} is minimized at $\mathrm{Id},$ that is $S_{\mu,K}$ is isotropic. Then, 
$$\left(\frac{n\kappa_n}{\kappa_{n-1}}\right)^n\kappa_n\leq\mu^n\left(\partial K\right)\vol_n(\Pi^\circ_\mu K)\leq \frac{4^nn^n}{n!}.$$
\end{theorem}

Using Theorem~\ref{t:almost_iso}, we obtain the following reverse isoperimetric inequalities for $\mu(\partial K).$ Theorem~\ref{t:qiso} below corresponds to measures with concavity similar to the logarithm, and follows from Corollary~\ref{Q_theorem}. Theorem~\ref{t:fiso} below corresponds to measures with concavity similar to $s$-concave measures, and follows from Corollary~\ref{F_theorem}.

\begin{theorem}
\label{t:qiso}
Let $\mu\in\Lambda$ be $Q$-concave, such that $Q:(0,\infty)\to\R$ is an invertible, increasing, differentiable function satisfying $\lim_{r\to0^+}Q(r)\in[-\infty,\infty).$ Suppose $K\in\conbod$ such that $K$ is $\mu$-projective and $Q^\prime (\mu(K))\neq 0$. Suppose further that $S_{\mu,K}$ is isotropic. Then,
$$\mu^n(\partial K)\leq \left(\frac{4n}{Q^\prime(\mu(K))}\right)^n\frac{\int_0^{\infty}Q^{-1}\left(Q(\mu(K))-t\right)t^{n-1}dt}{(n-1)!\vol_n(K)\mu(K)}.$$

\noindent In particular, if $\mu$ is log-concave, then one has
$$\mu(\partial K)\leq 4n \mu(K)\vol^{-1/n}_n(K).$$

\end{theorem}

\begin{theorem}
\label{t:fiso}
Let $\mu\in\Lambda$ be $F$-concave, such that $F:(0,\mu(\R^n))\to\R^+$ is an invertible, increasing, differentiable function. Suppose $K\in\conbod$ such that $K$ is $\mu$-projective. Suppose further that $S_{\mu,K}$ is isotropic. Then,
$$\mu^n(\partial K)\leq \left(\frac{4nF(\mu(K))}{F^\prime(\mu(K))}\right)^n\frac{\int_0^{1}F^{-1}\left(F(\mu(K))t\right)(1-t)^{n-1}dt}{(n-1)!\vol_n(K)\mu(K)}.$$

\noindent In particular, if $\mu$ is $s$-concave, for $s>0$, one has
$$\mu(\partial K)\leq 4n\left(\frac{\Gamma(\frac{1}{s}+1)}{s^n\Gamma(\frac{1}{s}+n+1)}\right)^{1/n} \mu(K)\vol^{-1/n}_n(K).$$

\end{theorem}

\printbibliography

\begin{tabular}{ll}
Dylan Langharst and Artem Zvavitch & Michael Roysdon\\
Department of Mathematical Sciences & Department of Pure Mathematics\\
Kent State University & Tel Aviv University\\
Kent, OH 44242 & P.O. Box 39040, Tel Aviv 6997801\\
USA & Israel\\
  dlanghar@kent.edu & mroysdon@kent.edu\\
   zvavitch@math.kent.edu
\end{tabular}

\end{document}